\documentclass[10pt,leqno,letterpaper]{article}
\usepackage[T1]{fontenc}
\usepackage[utf8]{inputenc}
\usepackage{amsmath}
\usepackage{amsfonts}
\usepackage{amssymb}
\usepackage{graphicx}
\usepackage{verbatim}
\usepackage{mathrsfs}
\usepackage{upref,amsthm,amsxtra,exscale}
\usepackage{nameref}
\usepackage{varioref}
\usepackage[colorlinks=true,urlcolor=blue,
citecolor=red,linkcolor=blue,linktocpage,pdfpagelabels,
bookmarksnumbered,bookmarksopen]{hyperref}
\usepackage{cleveref}
\usepackage{cite}
\usepackage{fullpage}
\usepackage[dvipsnames]{xcolor}
\newtheorem{theorem}{Theorem}[section]
\newtheorem{corollary}[theorem]{Corollary}

\newtheorem{lemma}[theorem]{Lemma}
\newtheorem{proposition}[theorem]{Proposition}
\newtheorem{definition}[theorem]{Definition}

\numberwithin{equation}{section}

\def\r{\mathbb{R}}
\def\rn{\mathbb{R}^N}
\def\z{\mathbb{Z}}
\def\zm{\z/m\z}
\def\zl{\z/\ell\z}
\def\s1{\mathbb{S}^1}
\def\n{\mathbb{N}}
\def\cc{\mathbb{C}}
\def\eps{\varepsilon}
\def\rh{\rightharpoonup}

\def\irn{\int_{\r^N}}
\def\vp{\varphi}

\def\vr{\varrho}
\def\o{\Omega}

\def\bf{\boldsymbol}

\def\cC{\mathcal{C}}

\def\cH^G{\mathcal{H}}

\def\cJ{\mathcal{J}}

\def\cM{\mathcal{M}}
\def\cN{\mathcal{N}}

\def\cP{\mathcal{P}}

\def\supp{\text{supp}}
\def\bar{\overline}
\def\what{\widehat}
\def\d{\,\mathrm{d}}
\def\e{\mathrm{e}}

\author{Mónica Clapp\footnote{M. Clapp was supported by CONAHCYT (Mexico) through the research grant A1-S-10457.}, \ Alberto Saldaña\footnote{A. Saldaña was supported by UNAM-DGAPA-PAPIIT (Mexico) grant IA100923  and by CONAHCYT (Mexico) grant A1-S-10457.}, \ Mayra Soares and Víctor A. Vicente-Benítez\footnote{V. A. Vicente-Ben\'itez was supported by CONAHCYT (Mexico) through the research program {\it Postdoctoral Stays for Mexico 2023}}.}
\title{Optimal pinwheel partitions and pinwheel solutions to a nonlinear Schrödinger system}
\date{}

\begin{document}
\maketitle

\begin{abstract}
We establish the existence of a solution to a nonlinear competitive Schrödinger system whose scalar potential tends to a positive constant at infinity with an appropriate rate. This solution has the property that all components are invariant under the action of a group of linear isometries and each component is obtained from the previous one by composing it with some fixed linear isometry. We call it a \emph{pinwheel solution}. We describe the asymptotic behavior of the least energy pinwheel solutions when the competing parameter tends to zero and to minus infinity. In the latter case the components are segregated and give rise to an optimal \emph{pinwheel partition} for the Schrödinger equation, that is, a partition formed by invariant sets that are mutually isometric through a fixed isometry. \\


\end{abstract}

\section{Introduction}

Consider the weakly coupled competitive Schrödinger system
\begin{equation} \label{eq:system}
\begin{cases}
-\Delta u_i + V(x)u_i= |u_i|^{2p-2}u_i+\sum\limits_{\substack{j=1\\j\neq i}}^\ell\beta|u_j|^p|u_i|^{p-2}u_i, \\
u_i\in H^1(\rn),\qquad i=1,\ldots,\ell,
\end{cases}
\end{equation}
where $N\geq 4$, $\beta<0$, $p\in(1,\frac{N}{N-2})$ and $V\in\cC^0(\rn)$. 

This kind of systems model interesting physical phenomena, for instance, the behavior of a mixture of Bose-Einstein condensates that overlap in space \cite{esry,dhall}. The existence and behavior of their solutions raises interesting and challenging questions from a purely mathematical point of view. They have been the subject of extensive study, particularly when the nonlinearity is cubic, $N=2,3$ and the scalar potential $V$ is constant. There is abundant literature on the matter. We refer to the article by Li, Wei, and Wu \cite{lww} for a list of references.

It is well known that the Schrödinger equation
\begin{equation}\label{eq:single_intro}
-\Delta u + V(x)u = |u|^{2p-2}u,\qquad u\in H^1(\rn),
\end{equation}
has a least energy solution if the potential $V$ is constant and positive. If $V$ is bounded away from zero and has a positive limit at infinity, the existence of a least energy solution depends on the way this limit is approached. For instance, if $V$ approaches its limit at infinity from below with an appropriate rate, then a least energy solution exists; see for example \cite{c}.

Systems behave differently. As Lin and Wei show in \cite[Theorem 1]{lw}, when $V$ is constant the system \eqref{eq:system} does not have a least energy solution. On the other hand, Sirakov showed in \cite{si} that it does have a least energy radial solution (i.e., each component $u_i$ is radial). Looking for solutions whose components have other symmetries does not always lead to a positive answer. A necessary and sufficient condition for the existence of a symmetric least energy solution, for constant $V$, is given by \cite[Theorem 3.4]{cs}.

Here we look for solutions $\bf u=(u_1,\ldots,u_\ell)$ whose components are not only invariant under the action of some group $G$ of linear isometries of $\rn$, but each component $u_{i+1}$ is obtained from the previous one $u_i$ by composing it with a fixed linear isometry. Nontrivial solutions of this kind have been called \emph{pinwheel solutions} in \cite{cp,cfs}.

Our assumptions on the potential are as follows:
\begin{itemize}
\item[$(V_1)$] $V$ is radial,
\item[$(V_2)$] $0<\inf_{x\in\rn}V(x)$ \ and \ $V(x)\to V_\infty>0$ as $|x|\to\infty$,
\item[$(V_3^m)$] There exist $\kappa\in(2\sin\frac{\pi}{m},2)$ and $\bar{C}>0$ such that
$$V(x)-V_\infty\leq \bar{C}\e^{-\kappa\sqrt{V_\infty}|x|}\qquad\text{for every \ }x\in\rn$$
and some number $m\in\n$.
\end{itemize}
The role of the number $m$ in assumption $(V_3^m)$ is related to the kind of symmetries that we consider, which we now describe.

Let $O(M)$ denote the group of linear isometries of $\r^M$. In this work, the group $G$ is either the additive group $\zm:=\{0,\ldots,m-1\}$ of integers modulo $m$, or the group $G_m:=\z/m\z\times O(N-4)$, acting on a point $x=(z_1,z_2,y)\in\cc\times\cc\times\r^{N-4}\equiv\rn$ as
\begin{align*}
&jx:=(\e^{2\pi\mathrm{i}j/m}z_1,\e^{-2\pi\mathrm{i}j/m}z_2,y)\quad\text{if \ }j\in \z/m\z, \\
&gx:=(\e^{2\pi\mathrm{i}j/m}z_1,\e^{-2\pi\mathrm{i}j/m}z_2,\alpha y)\quad\text{if \ }g=(j,\alpha)\in \z/m\z\times O(N-4).
\end{align*}
The linear isometry $\vr_\ell$ relating the components is
\begin{equation*}
\vr_\ell x:=\left(\Big(\cos\frac{\pi}{\ell}\Big)z+\Big(\sin\frac{\pi}{\ell}\Big)\tau z,\,y\right)\quad\text{for \ }x=(z,y)\in\cc^2\times\r^{N-4},
\end{equation*}
where $\tau(z_1,z_2):=(-\overline{z}_2,\overline{z}_1)$ and $\overline{z}_i$ is the complex conjugate of $z_i$. The solutions of \eqref{eq:system} that we obtain satisfy the following two conditions 
\begin{itemize}
\item[$(S_1^G)$] $u_i$ is $G$-invariant for every $i=1,\ldots,\ell$,
\item[$(S_2)$] \label{S_2} $u_{i+1}=u_i\circ\vr_\ell$ for $i=1,\ldots,\ell-1$ and $u_1=u_\ell\circ\vr_\ell$,
\end{itemize}
for one of the two groups $G$ introduced above. To emphasize which group $G$ is being considered, a nontrivial solution of \eqref{eq:system} that satisfies $(S_1^G)$ and $(S_2)$ will be called a $G$\emph{-pinwheel solution}. 

Let $\mathfrak{c}_\infty$ be the least energy level of a solution to the equation
\begin{equation}\label{eq:singleeqVinftyintro}
-\Delta u + V_\infty u = |u|^{2p-2}u,\qquad u\in H^1(\rn).
\end{equation}
This energy level is achieved at a radial positive solution $\omega\in H^1(\rn)$, which is unique up to translations and sign. 

The following is our main existence result.

\begin{theorem} \label{thm:existence}
\begin{itemize}
\item[$(i)$]If $N=4$ or $N\geq 6$, $V$ satisfies $(V_1),(V_2),(V_3^m)$, $m$ is even  and $m>2\ell$, then the system \eqref{eq:system} has a least energy $\left(\z/m\z\times O(N-4)\right)$-pinwheel solution $\bf u=(u_1,\ldots,u_\ell)$ that satisfies
\begin{align*}
\frac{p-1}{2p}\irn(|\nabla u_i|^2+V(x)u_i^2)<m \mathfrak{c}_\infty\quad\text{for every \ }i=1,\ldots,\ell.
\end{align*}
\item[$(ii)$]If $N\geq 4$, $V=V_\infty$, $m$ is even and $m>2\ell$, then the system \eqref{eq:system} with $V=V_\infty$ has a least energy $\left(\z/m\z\right)$-pinwheel solution $\bf u=(u_1,\ldots,u_\ell)$ that satisfies
\begin{align*}
\frac{p-1}{2p}\irn(|\nabla u_i|^2+V_\infty u_i^2)<m\mathfrak{c}_\infty\quad\text{for every \ }i=1,\ldots,\ell.
\end{align*}
\end{itemize}
\end{theorem}

For the precise meaning of \emph{least energy $G$-pinwheel solution}, we refer to Section \ref{sec:variational_setting}. 

It is shown in \cite[Theorem 3.4]{cs} that, if $N\geq 5$, $V$ is constant and $G=\z/m\z$, a least energy solution satisfying only $(S_1^{G})$ but not $(S_2)$ does not exist. So this last condition is necessary.

Pinwheel symmetries were first considered by Wei and Weth in \cite{ww} with the aim of obtaining non-radial solutions of a Schrödinger system of two equations with cubic nonlinearity in dimensions $N=2,3$. They have also been used, for instance, in \cite{pv, pw, cp, cfs}.

When $V$ approaches its limit at infinity from below, with an appropriate rate, the existence of a least energy pinwheel solution was established in \cite{cp} under simpler symmetry assumptions. Those symmetries are not enough to guarantee existence when the potential approaches its limit from above, or even for constant potentials. Let us elaborate on this.

The proof of Theorem \ref{thm:existence} is variational. Solutions are sought that minimize a constrained problem for the energy functional of the system \eqref{eq:system}, having the above described symmetries, see \eqref{cv:eq}. The main difficulty, as usual, is the lack of compactness. We analyze in detail the behavior of minimizing sequences that converge to the least energy level and, using concentration compactness arguments, we find that there are two possible scenarios: either the sequence has a subsequence that converges to a least energy solution, or there is a subsequence whose components look like a sum of copies of the ground state $\omega$ of \eqref{eq:singleeqVinftyintro}, centered on the elements of the $G$-orbit of some sequence of points tending to infinity, see Theorem \ref{thm:splitting}. In the latter case, the least energy of the system \eqref{eq:system} is $m\ell \mathfrak{c_\infty}$. So, in order to exclude this case, we need to construct a test function whose energy is strictly less than this number. The obvious test function is the one suggested by the second scenario. Showing that its energy is less than $m\ell \mathfrak{c_\infty}$ requires fine estimates for the interactions between the copies of $\omega$.  Here it becomes relevant that the $G$-orbit of some point $\xi$ contains two points whose distance is strictly less than the distances between the $G$-orbits of any pair of points $\xi,\vr_\ell\xi,\ldots,\vr_\ell^{\ell-1}\xi$. This condition is not satisfied by the symmetries considered in \cite{cp}.

When the potential is constant ($V=V_\infty$), as in statement $(ii)$ of Theorem \ref{thm:existence}, the system is invariant under translations. So one cannot expect that every minimizing sequence contains a convergent subsequence. This is only true up to translation. On the other hand, less symmetries are required in this case and the argument can also be applied in dimension $N=5$.

Note that, if $\beta\leq\frac{-1}{\ell-1}$, no component of a nontrivial solution $(u_1,\ldots,u_\ell)$ of \eqref{eq:system} satisfying $(S_2)$ can be radial. In fact, it cannot happen that $u_i=u_i\circ\vr_\ell$ for some $i$, as this would imply that $u_1=\cdots=u_\ell=:u$ and, thus, $u$ would be a solution to the equation
\begin{equation*}
-\Delta u + V(x) u = (1+(\ell-1)\beta)|u|^{2p-2}u,\qquad u\in H^1(\rn),
\end{equation*}
which has only the trivial solution.

The following two results describe the asymptotic behavior of least energy $G$-pinwheel solutions as the competition parameter $\beta$ goes to zero. As one might expect, the components of the solutions decouple.

\begin{theorem} \label{thm:betatozero1}
Let $N=4$ or $N\geq 6$, and $G:=\z/\z_m\times O(N-4)$ with $m$ even, $m>2\ell$. Assume that $V$ satisfies $(V_1),(V_2),(V_3^m)$. Let $\beta_k \to 0^- $ and $\bf u_k = (u_{k,1},\ldots, u_{k,\ell})$ be a least energy $G$-pinwheel solution to \eqref{eq:system} with $\beta =\beta_k$ such that $u_{k,j}\geq 0$ for every $k\in\n$, $j=1,\ldots,\ell$. Then, after passing to a subsequence, $u_{k, j}\to u_{0, j}$ strongly in $H^1(\rn)$, $u_{0,1}, \ldots , u_{0,\ell}$ satisfy $(S_1^G)$ and $(S_2)$, and $u_{0, j}$ is a positive least energy solution to the problem
$$-\Delta u + V(x)u = |u|^{2p-2}u,\qquad u\in H^1(\rn)^G,$$
where $H^1(\rn)^G:=\{u\in H^1(\rn):u\text{ is }G\text{-invariant}\}$. Furthermore, 
$$\frac{p-1}{2p}\irn(|\nabla u_{0,j}|^2+V(x)u_{0,j}^2) <m\mathfrak{c}_\infty\qquad\text{for every \ }j=1,\ldots,\ell.$$
\end{theorem}

\begin{theorem} \label{thm:betatozero2}
Let $N\geq 4$, $G:=\z/\z_m$ with $m$ even, $m>2\ell$, and $V=V_\infty$. Let $\beta_k \to 0^- $ and $\bf v_k = (v_{k,1},\ldots,v_{k,\ell})$ be a least energy $G$-pinwheel solution to \eqref{eq:system} with $\beta =\beta_k$ satisfying $v_{k,j}\geq 0$ for every $k\in\n$, $j=1,\ldots,\ell$. Then, after passing to a subsequence,  there is a sequence $(y_k)$ in $\r^{N-4}$ such that, if we set
$$\bf u_k:=(u_{k,1},\ldots,u_{k,\ell})\quad\text{with \ }u_{k,j}(z,y):=v_{k,j}(z,y-y_k)\text{ \ for all \ }(z,y)\in\cc^2\times\r^{N-4},$$
then $u_{k,j}\to\omega$ strongly in $H^1(\rn)$ for every $j=1,\ldots,\ell$, where $\omega$ is the least energy positive radial solution to the problem
$$-\Delta u + V_\infty u = |u|^{2p-2}u,\qquad u\in H^1(\rn).$$
\end{theorem}

It is by now well known that, as the repulsive force $\beta$ increases, the components of least energy solutions to weakly coupled systems tend to segregate and to give rise to an optimal partition. This fact was first noticed by Conti, Terracini and Verzini \cite{ctv1,ctv2} and by Chang, Lin, Lin and Lin \cite{clll} and has been established in different contexts, as for example in \cite{cfs,cp,cpt}. Our last two results are concerned with this question. We show that, in our case, as $\beta\to-\infty$ the components of least energy solutions segregate and give rise to an optimal partition that is made up of $G$-invariant sets that are mutually isometric via the isometry $\vr_\ell$. We call it an \emph{optimal $(G,\vr_\ell)$-pinwheel partition}. The precise notion is given in Definition \ref{def:partition}.

\begin{theorem}\label{theorembetainfinity1}
Let $N=4$ or $N\geqslant 6$ and $G=\mathbb{Z}/\mathbb{Z}_m\times O(N-4)$ with $m$ even and $2\sin\frac{\pi}{m}<\sin\frac{\pi}{2\ell}$. Assume that $V$ satisfies $(V_1), (V_2), (V_3^m)$. Let $\beta_k\rightarrow -\infty$ and $\bf u_k=(u_{k,1},\dots, u_{k,\ell})$ be a least energy $G$-pinwheel solution to \eqref{eq:system} with $\beta=\beta_k$ satisfying $u_{k,j}\geq 0$ for every $k\in\n$, $j=1,\ldots,\ell$. Then, after passing to a subsequence,
\begin{itemize}
		\item[$(i)$] $u_{k,j}\rightarrow u_{\infty,j}$ strongly in $H^1(\mathbb{R}^N)$, $u_{\infty,j}\geq 0$, $u_{\infty,j}\neq 0$ for each $j=1,\dots, \ell$, \ $u_{\infty,1}, \dots, u_{\infty,\ell}$ satisfy $(S_1^G)$ and $(S_2)$, and 
		\begin{align*}
u_{\infty,i}u_{\infty,j}=0 \qquad\text{and}\qquad	\lim_{k\to\infty}\int_{\rn}	\beta_ku_{k,j}^pu_{k,i}^p= 0 \qquad \mbox{whenever }\;  i\neq j.
		\end{align*}
	\item[$(ii)$] $u_{\infty,j}\in \mathcal{C}^0(\mathbb{R}^N)$, the restriction of $u_{\infty,j}$ to the set $\Omega_j=\{x\in \mathbb{R}^N: u_{\infty,j}(x)>0\}$ solves
$$-\Delta u + V(x)u = |u|^{2p-2}u,\qquad u\in H^1_0(\o_j)^G,$$
where $H^1_0(\o_j)^G:=\{u\in H_0^1(\o_j):u\text{ is }G\text{-invariant}\}$, and $(\Omega_1,\dots, \Omega_{\ell})$ is an optimal $(G,\vr_\ell)$-pinwheel partition for the problem
$$-\Delta u + V(x)u = |u|^{2p-2}u,\qquad u\in H^1(\rn)^G.$$
	\item[$(iii)$] $\mathbb{R}^N\smallsetminus \bigcup_{j=1}^{\ell}\Omega_j=\mathscr{R}\cup \mathscr{S}$ where $\mathscr{R}\cup \mathscr{S}=\emptyset$, $\mathscr{R}$ is a $\mathcal{C}^{1,\alpha}$-submanifold of $\mathbb{R}^N$ and $\mathscr{S}$ is a closed subset of $\mathbb{R}^N$ with Hausdorff measure $\leq N-2$. Furthermore, if $\xi \in \mathscr{R}$, there exists a pair $i,j$ such that
	\begin{align*}
		\lim_{x\to\xi^+}|\nabla u_{\infty,i}(x)|=\lim_{x\to\xi^{-}}|\nabla u_{\infty,j}(x)|\neq 0,
	\end{align*}
where $x\to \xi^\pm$ are the limits taken from the opposite sides of $\mathscr{R}$ and, if $\xi \in \mathscr{S}$, then 
\begin{align*}
	\lim_{x\rightarrow \xi}|\nabla u_j(x)|=0 \quad \mbox{ for each }\, j=1,\dots, \ell.
\end{align*}
\item[$(iv)$] If $\ell=2$, then $u_{\infty,1}-u_{\infty,2}$ is a $G$-invariant sign-changing solution of problem
$$-\Delta u + V(x)u = |u|^{2p-2}u,\qquad u\in H^1(\rn)^G.$$
\end{itemize}
\end{theorem}

\begin{theorem}\label{theorembetainfinity2}
Let $N\geq 4$, $G:=\z/m\z$ with $m$ even, $2\sin\frac{\pi}{m}<\sin\frac{\pi}{2\ell}$, and $V=V_\infty$. Let $\beta_k \to -\infty $ and let  $\bf v_k = (v_{k,1},\ldots,v_{k,\ell})$ be a least energy $G$-pinwheel solution to \eqref{eq:system} with $\beta =\beta_k$ such that $v_{k,j}\geq 0$ for all $j=1,\ldots\ell$. Then, after passing to a subsequence,  there exists a sequence $(y_k)$ in $\r^{N-4}$ such that, if we set
	\begin{align*}
		\bf u_k:=(u_{k,1},\ldots,u_{k,\ell})\qquad\text{with \ }u_{k,j}(z,y):=v_{k,j}(z,y-y_k)\text{ \ for all \ }(z,y)\in\cc^2\times\r^{N-4},
	\end{align*}
then the functions $u_{k,j}$ satisfy the statements $(i)-(iv)$ of \emph{Theorem \ref{theorembetainfinity1}}.
\end{theorem}

To obtain Theorems \ref{theorembetainfinity1} and \ref{theorembetainfinity2} we need an upper bound for the energy of the pinwheel solutions that is uniform in $\beta$. This requires having even better control of the interactions between the components of the test function. Here is where the slightly stronger assumption that $2\sin\frac{\pi}{m}<\sin\frac{\pi}{2\ell}$ comes into play, see Lemma \ref{lem:upperbound}.

The sign-changing solution to the Schrödinger equation given by Theorem \ref{theorembetainfinity1}$(iv)$ was already obtained in \cite{cd}, where more general symmetries are also considered.

The paper is structured as follows. Section \ref{sec:variational_setting} presents the symmetric variational framework. In Section \ref{sec:behaviorofmini} we describe the behavior of minimizing pinwheel sequences for the system \eqref{eq:system}. The proof of Theorem \ref{thm:existence} is given in Section \ref{sec:existence}. Section \ref{sec:limitprofileweak}  is devoted to the proofs of Theorems \ref{thm:betatozero1} and \ref{thm:betatozero2}, and Section \ref{sec:limitprofilestrong} to those of Theorems \ref{theorembetainfinity1} and \ref{theorembetainfinity2}.

\section{The variational setting}
\label{sec:variational_setting}

Assume that $V$ satisfies $(V_1),(V_2)$. Then,
$$\langle u,v\rangle_V:=\int_{\rn}\left(\nabla u\cdot\nabla v+V(x)uv\right)\qquad\text{and}\qquad\|u\|_V:=\sqrt{\langle u,u\rangle}$$
are an inner product and a norm in $H^1(\rn)$, equivalent to the standard one. We write $|u|_{2p}$ for the norm of $u$ in $L^{2p}(\rn)$. The solutions $\bf u=(u_1,\ldots,u_\ell)$ to \eqref{eq:system} are the critical points of the functional $\mathcal{J}_V:(H^1(\rn))^\ell\to\r$ given by 
\[\mathcal{J}_V(\bf u):= \frac{1}{2}\sum_{i=1}^\ell\|u_i\|_V^2 - \frac{1}{2p}\sum_{i=1}^\ell |u_i|_{2p}^{2p}-\frac{\beta}{2p}\sum_{\substack{i,j=1 \\ i\neq j}}^\ell\irn |u_i|^p|u_j|^p,\]
which is of class $\cC^1$. Its $i\text{-th}$ partial derivative is
$$\partial_i\mathcal{J}_V(\bf u)v=\langle u_i,v\rangle_V-\irn|u_i|^{2p-2}u_iv-\beta\sum\limits_{\substack{j=1\\j\neq i}}^\ell\irn|u_j|^p|u_i|^{p-2}u_iv$$
for any $\bf u\in(H^1(\rn))^\ell, \ v\in H^1(\rn)$. Next, we describe the symmetric variational setting to obtain $G$-pinwheel solutions.

\subsection{The action on each component}

Let $m\in\n$ be \emph{an even number} (the need for this assumption will become clear in the proof of Proposition \ref{prop:vr_ell}) and consider the additive group $\zm:=\{0,\ldots,m-1\}$ of integers modulo $m$. For each $j\in\zm$ define $\vartheta_m^j:\cc^2\to\cc^2$ as
\begin{equation}\label{eq:zm_action}
\vartheta_m^jz:=(\e^{2\pi\mathrm{i}j/m}z_1,\e^{-2\pi\mathrm{i}j/m}z_2)\quad\text{for every \ }z=(z_1,z_2)\in\cc\times\cc.	
\end{equation}
This gives an action of $\z/m\z$ on $\rn$ defined by
$$jx:=(\vartheta_m^jz,y)\quad\text{for every \ }j\in \z/m\z\text{ \ and \ }x=(z,y)\in\cc^2\times\r^{N-4}\equiv\rn.$$	
We consider also the group $G_m:=\z/m\z\times O(N-4)$ acting on $\rn$ by
$$gx:=(\vartheta_m^jz,\alpha y)\quad\text{for every \ }g=(j,\alpha)\in \z/m\z\times O(N-4)\text{ \ and \ }x=(z,y)\in\cc^2\times\r^{N-4}\equiv\rn,$$
where, as usual, $O(M)$ denotes the group of linear isometries of $\r^M$.

Hereafter \emph{$G$ will denote either one of these two groups}. If $g\in G$ and $u\in H^1(\rn)$ we define $gu(x):=u(g^{-1}x)$. Since $V$ is radial, $\|gu\|_V=\|u\|_V$. So setting
$$g\bf u:=(gu_1,\ldots,gu_\ell)\quad\text{for \ } \bf u=(u_1,\ldots,u_\ell)\in (H^1(\rn))^\ell\text{ \ and \ }g\in G,$$
we obtain an isometric action of $G$ on $(H^1(\rn))^\ell$. The $G$-fixed point space of $(H^1(\rn))^\ell$ is the subspace,
\begin{align*}
\mathcal{H}^G:&=\{\bf u\in (H^1(\rn))^\ell:g\bf u=\bf u \text{ for all }g\in G\}\\
&=\{\bf u\in (H^1(\rn))^\ell:u_i\text{ is }G\text{-invariant for every }i=1,\ldots,\ell\}.
\end{align*}
Note that the functional $\mathcal{J}_V$ is $G$-invariant. So, by the principle of symmetric criticality \cite[Theorem 1.28]{W}, the critical points of the restriction of $\mathcal{J}_V$ to $\mathcal{H}^G$ are the solutions to the system \eqref{eq:system} whose components are $G$-invariant.

\subsection{The action on the set of components}

Let $\tau:\cc^2\to\cc^2$ be given by $\tau(z_1,z_2):=(-\overline{z}_2,\overline{z}_1)$ and, for each $n\in\zl=\{0,\ldots,\ell-1\}$, define $\vr_\ell^n:\rn\to\rn$ by
\begin{equation}\label{eq:eq:vrj}
\vr_\ell^nx:=\left(\Big(\cos\frac{\pi n}{\ell}\Big)z+\Big(\sin\frac{\pi n}{\ell}\Big)\tau z,\,y\right)\qquad\text{for every \ }x=(z,y)\in\cc^2\times\r^{N-4}.
\end{equation}
Since $\tau z$ is orthogonal to $z$ and $|\tau z|=|z|$ we have that $\vr_\ell^n\in O(N)$. Furthermore, as $\tau\vartheta_m^j=\vartheta_m^j\tau$ for every $j\in\zm$, we have that $\vr_\ell^ng=g\vr_\ell^n$ for every $g\in G$ and $n\in\zl$.

We denote by $\sigma^n:\{1,\ldots,\ell\}\to\{1,\ldots,\ell\}$ the permutation $\sigma^n(r):=r+n \mod\ell$, \ $n\in\zl$.

\begin{proposition} \label{prop:vr_ell}
\begin{itemize}
\item[$(a)$] For each $n\in\zl$ the function $\vr_\ell^n:\mathcal{H}^G\to\mathcal{H}^G$ given by
$$\vr_\ell^n\bf u(x):=(u_{\sigma^n(1)}(\vr_\ell^{-n}x),\ldots,u_{\sigma^n(\ell)}(\vr_\ell^{-n}x)),\quad\text{where \ }\bf u=(u_1,\ldots,u_\ell),$$
is well-defined and is a linear isometry.
\item[$(b)$] $n\mapsto\vr_\ell^n$ is a well-defined action of $\zl$ on $\mathcal{H}^G$.
\end{itemize}
\end{proposition}

\begin{proof}
$(a):$ \ Let $\bf u\in\mathcal{H}^G$, $g\in G$ and $n\in\zl$. Since each $u_i$ is $G$-invariant and $\vr_\ell^ng=g\vr_\ell^n$ we have
\begin{align*}
g[\vr_\ell^n\bf u](x)&=[\vr_\ell^n\bf u](g^{-1}x)\\
&=(u_{\sigma^n(1)}(\vr_\ell^{-n}g^{-1}x),\ldots,u_{\sigma^n(\ell)}(\vr_\ell^{-n}g^{-1}x))\\
&=(u_{\sigma^n(1)}(g^{-1}\vr_\ell^{-n}x),\ldots,u_{\sigma^n(\ell)}(g^{-1}\vr_\ell^{-n}x))\\
&=(u_{\sigma^n(1)}(\vr_\ell^{-n}x),\ldots,u_{\sigma^n(\ell)}(\vr_\ell^{-n}x))=\vr_\ell^n\bf u(x). 
\end{align*}
This shows that $\vr_\ell^n\bf u\in\mathcal{H}^G$. The function $\vr_\ell^n:\mathcal{H}^G\to\mathcal{H}^G$ is clearly linear and bijective and, since $V$ is radial, it satisfies $\|\vr_\ell^n\bf u\|_V=\|\bf u\|_V$.

$(b):$ \ $\vr_\ell^{-\ell}(z,y)=(-z,y)$ for every $(z,y)\in\cc^2\times\r^{N-4}$. So, if $\bf u\in\mathcal{H}^G$ then, since $m$ is even,
\begin{align*}
\vr_\ell^\ell\bf u(z,y)&=(u_{\sigma^\ell(1)}(\vr_\ell^{-\ell}( z,y)),\ldots,u_{\sigma^\ell(\ell)}(\vr_\ell^{-\ell}(z,y)))\\
&=(u_1(-z,y),\ldots,u_\ell(-z,y))=\bf u(-z,y)=\bf u(\vartheta_m^\frac{m}{2}z,y)=\bf u(z,y).
\end{align*}
This shows that $\vr_\ell^{\ell}:\mathcal{H}^G\to\mathcal{H}^G$ is the identity. So $j\mapsto\vr_\ell^j$ is a well-defined homomorphism from $\zl$ into the group of linear isometries of $\mathcal{H}^G$.
\end{proof}

The $\zl$-fixed point space of $\mathcal{H}^G$ is the space
\begin{align*}
\mathscr{H}^G:&=(\mathcal{H}^G)^{\zl}=\{\bf u\in\mathcal{H}^G:\vr_\ell^n\bf u=\bf u\text{ \ for all \ }n\in\zl \} \\
&=\{\bf u\in (H^1(\rn))^\ell:u_i\text{ is }G\text{-invariant, }u_{i+1}=u_i\circ\vr_\ell\text{ for }i=1,\ldots,\ell-1\text{ and }u_1=u_\ell\circ\vr_\ell\}.
\end{align*}
The functional $\mathcal{J}_V|_{\mathcal{H}^G}:\mathcal{H}^G\to\r$ is $\zl$-invariant. So, by the principle of symmetric criticality, the critical points of its restriction to $\mathscr{H}^G$ are precisely the solutions of \eqref{eq:system} that satisfy the conditions $(S_1^G)$ and $(S_2)$ stated in the introduction, called $G$-pinwheel solutions. Abusing notation, we write
\begin{equation}\label{eq:energy}
\mathcal{J}_V:=\mathcal{J}_V|_{\mathscr{H}^G}:\mathscr{H}^G\to\r.
\end{equation}
Note that
\begin{align*}
\mathcal{J}_V'(\bf u)\bf v=\sum_{i=1}^\ell\partial_i\mathcal{J}_V(\bf u)v_i=\ell\,\partial_j\mathcal{J}_V(\bf u)v_j\quad\text{for any \ }\bf u,\bf v\in\mathscr{H}^G\text{ and }j=1,\ldots,\ell.
\end{align*}
If $\bf u\in\mathscr{H}^G$ and $\bf u\neq 0$ then, by condition $(S_2)$, every component of $\bf u$ is nontrivial. So the fully nontrivial critical points of $\mathcal{J}_V|_{\mathscr{H}^G}$ belong to the Nehari manifold
\begin{align}\label{eq:nehari}
\mathcal{N}_V^G=\{\bf u\in\mathscr{H}^G:\bf u\neq 0, \ \mathcal{J}_V'(\bf u)\bf u=0\}.
\end{align}
Set
\begin{align}\label{cv:eq}
c_V^G:=\inf_{\bf u\in\mathcal{N}_V^G}\mathcal{J}_V(\bf u).    
\end{align}
A function $\bf u\in\mathcal{N}_V^G$ such that $\mathcal{J}_V(\bf u)=c_V^G$ is called a \emph{least energy $G$-pinwheel solution} to \eqref{eq:system}.

We consider also the single equation
\begin{equation} \label{eq:single}
-\Delta u + V(x)u = |u|^{2p-2}u,\qquad u\in H^1(\rn)^G,
\end{equation}
where $H^1(\rn)^G = \{u\in H^1(\rn) : u \mbox{ is $G$- invariant} \}$ and we denote by $J:H^1(\rn)^G\to\r$ the energy functional and by and $\cM^G$ the Nehari manifold associated to it, i.e.,
\begin{equation}\label{functionalJsingle}
J(u):=\frac{1}{2}\|u\|_V^2 - \frac{1}{2p}\irn |u|^{2p}
\end{equation}
and
$$\cM^G:=\left\{u\in H^1(\rn)^G:u\neq 0, \ \|u\|_V^2 =\irn |u|^{2p}\right\}.$$
Similarly, we write $J_\infty:H^1(\rn)\to\r$ and $\cM_\infty$ for the energy functional and the Nehari manifold of
\begin{equation} \label{eq:limit_problem}
-\Delta u + V_\infty u = |u|^{2p-2}u,\qquad u\in H^1(\rn),
\end{equation}
and set 
\begin{equation}\label{eq:frak_c}
\mathfrak{c}_\infty:=\inf_{u\in\cM_\infty}J_\infty(u)\qquad\text{and}\qquad\mathfrak{c}^G:=\inf_{u\in\cM^G}J(u).
\end{equation}

\begin{proposition} \label{prop:nehari}
\begin{itemize}
\item[$(a)$] $\mathcal{N}_V^G\neq\emptyset$. 
\item[$(b)$] There exists $a_0>0$, independent of $\beta$, such that $\|\bf u\|_V^2\geq a_0$ for every $\bf u\in\cN_V^G$.
\item[$(c)$] $\mathcal{N}_V^G$ is a closed $\cC^1$-submanifold of codimension $1$ of $\mathscr{H}^G$, and a natural constraint for $\mathcal{J}_V$.
\item[$(d)$] If $\bf u\in\mathscr{H}^G$ is such that, for each $i=1,\ldots,\ell$,
\begin{equation} \label{eq:proj_nehari}
\irn|u_i|^{2p} + \sum_{\substack{j=1 \\ j\neq i}}^\ell\beta\irn|u_i|^p|u_j|^p>0,
\end{equation}
then there exists a unique $s_{\bf u} \in (0,\infty)$ such that \ $s_{\bf u}\bf u\in \mathcal{N}_V^G$. Furthermore,  
\begin{align}\label{eq:functionalJs}
	\mathcal{J}_V(s_{\bf u}\bf u) = \dfrac{p-1}{2p} \left(\frac{\sum\limits_{i=1}^\ell\|u_i\|_V^2}{\Big(\sum\limits_{i=1}^\ell|u_i|_{2p}^{2p}+\sum\limits_{\substack{i,j=1 \\ j\neq i}}^\ell\beta\irn|u_i|^p|u_j|^p\Big)^\frac{1}{p}}\right)^\frac{p}{p-1}.
\end{align}
\item[$(e)$] $c_V^G\leq \ell m\mathfrak{c_\infty}$.
\end{itemize}
\end{proposition}

\begin{proof}
The proof is straightforward and similar to that of \cite[Proposition 3.1]{cp}.
\end{proof}

\section{The behavior of minimizing sequences}\label{sec:behaviorofmini}

In this section, we describe the behavior of minimizing $G$-pinwheel sequences for \eqref{eq:system}. We assume throughout that $V$ satisfies $(V_1),(V_2)$. We need the following lemmas.

\begin{lemma} \label{lem:Vinfty}
Assume $v_k\rh v$ weakly in $H^1(\rn)$, $\xi_k\in\rn$ satisfies $|\xi_k|\to\infty$ and $V\in\cC^0(\rn)$ satisfies $(V_2)$. Set $V_k(x):=V(x+\xi_k)$. Then,
$$\lim_{k\to\infty}\|v_k\|_{V_k}^2-\lim_{k\to\infty}\|v_k-v\|_{V_k}^2=\|v\|_{V_\infty}^2.$$
\end{lemma}

\begin{proof}
See \cite[Lemma A.1]{cp}.
\end{proof}

\begin{lemma} \label{lem:infinite orbit}
Let $M\geq 2$. Given an unbounded sequence $(y_k)$ in $\r^M$ and $n\in\n$, after passing to a subsequence, there exist $\alpha_1,\ldots,\alpha_n\in O(M)$ such that $|\alpha_iy_k-\alpha_jy_k|\to\infty$ as $k\to\infty$ whenever $i\neq j$.
\end{lemma}

\begin{proof} Choose a subsequence of $(y_k)$ such that $y_k\neq 0$ and $\frac{y_k}{|y_k|}\to y$ in $\r^{M}$.
Since $M\geq 2$, for any given $n\in\n$, there exist $\alpha_1,\ldots,\alpha_n\in O(M)$ such that $\alpha_iy\neq\alpha_jy$ whenever $i\neq j$. Therefore, there exist $d>0$ and $k_0\in\n$ such that
$$\left|\alpha_i\frac{y_k}{|y_k|}-\alpha_j\frac{y_k}{|y_k|}\right|\geq d\qquad\text{if \ }k\geq k_0\text{ \ and \ }i\neq j.$$
It follows that $|\alpha_iy_k-\alpha_jy_k|\geq d|y_k|\to\infty$, as claimed.
\end{proof}

\begin{theorem} \label{thm:splitting} 
Let $N=4$ or $N\geq 6$, $G:=\z/m\z\times O(N-4)$ and $\bf u_k=(u_{k,1},\ldots,u_{k,\ell})\in\mathcal{N}_V^G$ be such that $\mathcal{J}_V(\bf u_k)\to c_V^G$ and $u_{k,i}\geq 0$. Then, after passing to a subsequence,
\begin{itemize}
\item[$(I)$] either $\bf u_k\to\bf u=(u_1,\ldots,u_\ell)$ strongly in $\mathscr{H}^G$ and $u_i\geq 0$,
\item[$(II)$] or there are points $(\zeta_k,0)\in\cc^2\times\r^{N-4}\equiv\rn$ such that $|\zeta_{k}|\to\infty$,
$$\lim_{k\to\infty}\Big\|u_{k,1}-\sum_{j=0}^{m-1}\omega\big( \ \cdot \ -\,(\vartheta_m^j\zeta_k,0)\big)\Big\|=0,$$
and \ $c_V^G = \ell m\mathfrak{c}_\infty$, where $\omega$ is the least energy positive radial solution to \eqref{eq:limit_problem} and $\mathfrak{c}_\infty$ is its energy \eqref{eq:frak_c}.
\end{itemize}
\end{theorem}

\begin{proof}
The proof is similar to that of \cite[Theorem 3.4]{cp}. We include it for completeness. Using Ekeland's variational principle \cite[Theorem 8.5]{W}, we can assume that $\mathcal{J}_V'(\bf u_k)\to 0$ in $(\mathscr{H}^G)'$. Since $\beta<0$ we derive from Proposition \ref{prop:nehari}$(b)$ that there exists $a_0>0$ such that
$$\irn|u_{k,1}|^{2p}>a_0\qquad\forall k\in\n.$$
Then, by Lions' lemma \cite[Lemma 1.21]{W}, there exist $\delta>0$ and $x_k=(z_k,y_k)\in \mathbb{C}^2\times\r^{N-4}\equiv\rn$ such that, after passing to a subsequence,
\begin{equation*}
\int_{B_1(x_k)}|u_{k,1}|^{2p}>\delta\qquad\forall k\in\n.
\end{equation*}
We choose a subsequence of $(x_k)$ having one of the following three properties:
\begin{itemize}
\item[$(1)$] either $|x_k|\leq C_0$ for all $k\in\n$ and some $C_0>0$,
\item[$(2)$] or $|y_k|\leq C_0$ for all $k\in\n$ and some $C_0>0$ and $|z_k|\to\infty$, 
\item[$(3)$] or $y_k\neq 0$ for all $k\in\n$ and $|y_k|\to\infty$. 
\end{itemize}
In the first case we define $\xi_k:=0$, in the second one we set $\xi_k:=(z_k,0)$ and in the third one we take $\xi_k:=x_k$. We define
$$w_{k,1}(x):=u_{k,1}(x+\xi_k).$$
Then, in all three cases, $|x_k-\xi_k|\leq C_0$ and
\begin{equation} \label{eq:nonnull}
\int_{B_{C_0+1}(0)}|w_{k,1}|^{2p}=\int_{B_{C_0+1}(\xi_k)}|u_{k,1}|^{2p}\geq\int_{B_1(x_k)}|u_{k,1}|^{2p}>\delta>0\qquad\forall k\in\n.
\end{equation}
Now we analyze each one of these cases.

First, we show that $(3)$ cannot occur. Indeed, since in this case $N\geq 6$, Lemma \ref{lem:infinite orbit} states that, for any given $n\in\n$ there exist $\alpha_1,\ldots,\alpha_n\in O(N-4)$ such that $|\alpha_ix_k-\alpha_jx_k|=|\alpha_iy_k-\alpha_jy_k|\geq 2$ for large enough $k$ and $i\neq j$. As $u_{k,1}$ is $G$-invariant, equation \eqref{eq:nonnull} yields
\begin{equation*}
\irn|u_{k,1}|^{2p}\geq\sum_{j=1}^n\int_{B_1(\alpha_jx_k)}|u_{k,1}|^{2p}=n\int_{B_1(x_k)}|u_{k,1}|^{2p}>n\delta\qquad\forall n\in\n,
\end{equation*}
contradicting the fact that $(u_{k,1})$ is bounded in $L^{2p}(\rn)$.
\smallskip  

Assume $(1).$ \ Then $\xi_k=0$. Passing to a subsequence, we have that $u_{k,i}\rh u_i$ weakly in $H^1(\rn)$, $u_{k,i}\to u_i$ in $L^{2p}_\mathrm{loc}(\rn)$ and $u_{k,i}\to u_i$ a.e. in $\rn$. Hence $\bf u=(u_1,\ldots,u_\ell)\in\mathscr{H}^G$, $u_i\geq 0$ and from \eqref{eq:nonnull} we get that $u_i\neq 0$. Since
\begin{align*}
0=\lim_{k\to\infty}\partial_1\mathcal{J}_V(\bf u_k)\vp=\partial_1\mathcal{J}_V(\bf u)\vp\qquad\text{for every \ }\vp\in\cC^\infty_c(\rn),
\end{align*}
we have that $\bf u\in\mathcal{N}_V^G$. Therefore,
\begin{align*}
c_V^G\leq\mathcal{J}_V(\bf u)=\frac{p-1}{2p}\sum_{i=1}^\ell\|u_i\|_V^2\leq\liminf_{k\to\infty}\frac{p-1}{2p}\sum_{i=1}^\ell\|u_{k,i}\|_V^2 =\lim_{k\to\infty}\mathcal{J}_V(\bf u_k)= c_V^G.
\end{align*}
As a consequence, $\bf u_k\to\bf u$ strongly in $\mathscr{H}^G$, i.e., statement $(I)$ holds true.
\smallskip

Assume $(2).$ \ Then $\xi_k = (z_k,0) \in \cc^2 \times \mathbb{R}^{N-4}$, $|z_k|\to\infty$ and $w_{k,1}$ is $O(N-4)$-invariant. Since the sequence $(w_{k,1})$ is bounded in $H^1(\rn)$, passing to a subsequence, $w_{k,1}\rh w_1$ weakly in $H^1(\rn)$, $w_{k,1}\to w_1$ in $L^{2p}_\mathrm{loc}(\rn)$ and $w_{k,1}\to w_1$ a.e. in $\rn$. Hence, $w_1$ is $O(N-4)$-invariant, $w_1\geq 0$ and, by \eqref{eq:nonnull}, $w_1\neq 0$. As $|\vartheta_m^j\xi_k-\vartheta_m^i\xi_k|\to\infty$ for any $i,j\in\z/m\z$ with $i\neq j$, we have that
$$w_{k,1}\circ\vartheta_m^{-j}-\sum_{i=j+1}^{m-1}(w_1\circ\vartheta_m^{-i})(\,\cdot\,-\vartheta_m^i\xi_k+\vartheta_m^j\xi_k)\rh w_1\circ\vartheta_m^{-j}\qquad\text{weakly in \ }H^1(\rn).$$ 
Setting $V_k(x):=V(x+\xi_k)$, Lemma \ref{lem:Vinfty} gives
\begin{align*}
&\Big\|w_{k,1}\circ\vartheta_m^{-j}-\sum_{i=j+1}^{m-1}(w_1\circ\vartheta_m^{-i})(\,\cdot\,-\vartheta_m^i\xi_k+\vartheta_m^j\xi_k)\Big\|_{V_k}^2\\
&\qquad=\Big\|w_{k,1}\circ\vartheta_m^{-j}-\sum_{i=j}^{m-1}(w_1\circ\vartheta_m^{-i})(\,\cdot\,-\vartheta_m^i\xi_k+\vartheta_m^j\xi_k)\Big\|_{V_k}^2+ \|w_1\circ\vartheta_m^{-j}\|_{V_\infty}^2 +o(1).
\end{align*}
Since $u_{k,1}$ is $G$-invariant, performing the change of variable $\widetilde{x}=x-\vartheta_m^j\xi_k$ yields
\begin{align*}
\Big\|u_{k,1}-\sum_{i=j+1}^{m-1}(w_1\circ\vartheta_m^{-i})(\,\cdot\,-\vartheta_m^i\xi_k)\Big\|_V^2
=\Big\|u_{k,1}-\sum_{i=j}^{m-1}(w_1\circ\vartheta_m^{-i})(\,\cdot\,-\vartheta_m^i\xi_k)\Big\|_V^2+\|w_1\|_{V_\infty}^2+o(1),
\end{align*}
and iterating this identity we obtain
\begin{equation} \label{eq:norms}
\|u_{k,1}\|_V^2=\Big\|u_{k,1}-\sum_{i=0}^{m-1}(w_1\circ\vartheta_m^{-i})(\,\cdot\,-\vartheta_m^i\xi_k)\Big\|_V^2+m\|w_1\|_{V_\infty}^2+o(1).
\end{equation}
Now, for any $v\in H^1(\rn)$ we set $v_{k}(x):=v(x-\xi_{k})$ and we define $w_{k,i}(x):=u_{k,i}(x+\xi_k)$. Then, a subsequence satisfies $w_{k,i}\rh w_i$ weakly in $H^1(\rn)$. Performing a change of variable we derive
$$o(1)=\partial_1\mathcal{J}_V(\bf u_k)v_k =\irn(\nabla w_{k,1}\cdot\nabla v+V_kw_{k,1}v)-\irn|w_{k,1}|^{2p-2}w_{k,1}v -\beta\sum_{j=2}^\ell\irn|w_{k,j}|^p|w_{k,1}|^{p-2}w_{k,1}v,$$
and passing to the limit as $k\to\infty$ we get
\begin{equation}\label{eq:projectionnehari}
0=\irn(\nabla w_1\cdot\nabla v+V_\infty w_1v)-\irn|w_1|^{2p-2}w_1v-\beta\sum_{j=2}^\ell\irn|w_j|^p|w_1|^{p-2}w_1v.
\end{equation}
Therefore,
\begin{equation*}
\|w_1\|_{V_\infty}^2=\irn|w_1|^{2p}+\beta\sum_{j=2}^\ell\irn|w_j|^p|w_1|^{p}\leq\irn|w_1|^{2p},
\end{equation*}
so there exists $t\in(0,1]$ such that $\|tw_1\|_{V_\infty}^2=\irn|tw_1|^{2p}$. It follows that $tw_1\in\cM_\infty$, and from equation \eqref{eq:norms} and Proposition \ref{prop:nehari}$(e)$ we derive
\begin{align*}
m\mathfrak{c}_\infty\leq\frac{p-1}{2p}m\|tw_1\|_{V_\infty}^2\leq\frac{p-1}{2p}m\|w_1\|_{V_\infty}^2\leq\lim_{k\to\infty}\frac{p-1}{2p}\|u_{k,1}\|_V^2=\frac{1}{\ell}c_V^G\leq m\mathfrak{c}_\infty.
\end{align*}
Therefore, $t=1$, $w_1\in\cM_\infty$ and $J_\infty(w_1)=\frac{p-1}{2p}\|w_1\|_{V_\infty}^2=\mathfrak{c}_\infty$, i.e., $w_1$ is a least energy solution of \eqref{eq:limit_problem}. Moreover, from \eqref{eq:norms} we get that
$$\lim_{k\to\infty}\Big\|u_{k,1}-\sum_{j=0}^{m-1}(w_1\circ\vartheta_m^{-j })(\,\cdot\,-\vartheta_m^j\xi_k)\Big\|_V^2=0.$$
Since the positive least energy solution to \eqref{eq:limit_problem} is unique up to translation and $w_1$ is $O(N-4)$-invariant, there exists $\xi=(\zeta,0)\in\cc^2\times\r^{N-4}$ such that $w_1(x)=\omega(x-\xi)$. Hence, $(w_1\circ\vartheta_m^{-j })(x-\vartheta_m^j\xi_k)=\omega(\vartheta_m^{-j}x-\xi_k-\xi)=\omega(x-\vartheta_m^j(\xi_k+\xi))$. So, setting $\zeta_k:=z_k+\zeta$, we obtain
$$\lim_{k\to\infty}\Big\|u_{k,1}-\sum_{j=0}^{m-1}\omega\big( \ \cdot \ -\,(\vartheta_m^j\zeta_k,0)\big)\Big\|_V=0.$$
This shows that statement $(II)$ holds true.
\end{proof}

\begin{corollary}\label{cor:compactness}
If $N=4$ or $N\geq 6$, $G:=\z/m\z\times O(N-4)$ and $c_V^G < \ell m\mathfrak{c}_\infty$, the system \eqref{eq:system} has a least energy $\left(\z/m\z\times O(N-4)\right)$-pinwheel solution.
\end{corollary}

When $V \equiv V_\infty$ the following statement holds true for every $N\geq 4$.

\begin{theorem} \label{thm:splitting2} 
Let $N\geq 4$, $G:=\z/m\z $ and $\bf v_k=(v_{k,1},\ldots,v_{k,\ell})\in\mathcal{N}_{V_\infty}^G$ be such that $\mathcal{J}_{V_\infty}(\bf v_k)\to c_{V_\infty}^G$ and $v_{k,i}\geq 0$. Then, after passing to a subsequence, there is a sequence $(y_k)$ in $\r^{N-4}$ such that if we set
$$\bf u_k:=(u_{k,1},\ldots,u_{k,\ell})\quad\text{with \ }u_{k,i}(z,y):=v_{k,i}(z,y-y_k)\text{ \ for all \ }(z,y)\in\cc^2\times\r^{N-4},$$
we have that $\bf u_k\in\mathcal{N}_{V_\infty}^G$, $\mathcal{J}_{V_\infty}(\bf u_k)\to c_{V_\infty}^G$ and
\begin{itemize}
\item[$(I)$] either $\bf u_k\to\bf u=(u_1,\ldots,u_\ell)$ strongly in $\mathscr{H}^G$ and $u_i\geq 0$,
\item[$(II)$] or there are points $(\zeta_k,\eta)\in\cc^2\times\r^{N-4}$ such that $|\zeta_{k}|\to\infty$,
$$\lim_{k\to\infty}\Big\|u_{k,1}-\sum_{j=0}^{m-1}\omega\big( \ \cdot \ -\,(\vartheta_m^j\zeta_k,\eta)\big)\Big\|_{V_\infty}=0,$$
and \ $c_{V_\infty}^G = \ell m\mathfrak{c}_\infty$.
\end{itemize}
\end{theorem}

\begin{proof}
Arguing as in the proof of Theorem \ref{thm:splitting} we see that there exist $\delta>0$ and $x_k=(z_k,y_k)\in \mathbb{C}^2\times\r^{N-4}\equiv\rn$ such that, after passing to a subsequence,
\begin{equation*}
\int_{B_1(x_k)}|v_{k,1}|^{2p}>\delta\qquad\forall k\in\n.
\end{equation*}
Then the functions $u_{k,i}$ defined by $u_{k,i}(z,y):=v_{k,i}(z,y-y_k)$ are $\z/m\z$-invariant and $u_{k,i+1}=u_{k,i}\circ\vr_\ell$, so $\bf u_k\in\mathscr{H}^{\z/m\z}$ and, since $\|u_{k,i}\|_{V_\infty}=\|v_{k,i}\|_{V_\infty}$, we have that $\bf u_k\in\mathcal{N}_{V_\infty}^G$ and $\mathcal{J}_{V_\infty}(\bf u_k)\to c_{V_\infty}^G$. Furthermore,
\begin{equation*}
\int_{B_1(z_k,0)}|u_{k,1}|^{2p}=\int_{B_1(x_k)}|v_{k,1}|^{2p}>\delta\qquad\forall k\in\n.
\end{equation*} 
After passing to a subsequence, there are only two cases:
\begin{itemize}
\item[$(1')$] either $|z_k|\leq C_0$ for all $k$ and some $C_0>0$,
\item[$(2')$] or $|z_k|\to\infty$.
\end{itemize}
In the first case we set $\xi_k:=0$, in the second one we take $\xi_k:=(z_k,0)$, and we define
$$w_{k,1}(x):=u_{k,1}(x+\xi_k).$$
From this point on, we just follow the proof of Theorem \ref{thm:splitting} (noting that this time there is no case $(3)$ and that the functions $w_{k,1}$ and $w_1$ need not be $O(N-4)$-invariant).
\end{proof}

\begin{corollary}\label{cor:compactness2}
If $N\geq 4$, $G:=\z/m\z $ and $c_{V_\infty}^G < \ell m\mathfrak{c}_\infty$, the system \eqref{eq:system} with $V\equiv V_\infty$ has a least energy $\left(\z/m\z\right)$-pinwheel solution.
\end{corollary}

\section{Existence of a least energy pinwheel solution}\label{sec:existence}

This section is devoted to the proof of Theorem \ref{thm:existence}. We assume throughout that $V$ satisfies $(V_1),(V_2),(V_3^m)$ and $G$ is either $\z/m\z$ or $\z/m\z\times O(N-4)$.

Let $\omega\in \cM_\infty$ be the unique least energy positive radial solution to \eqref{eq:limit_problem}. It is well known \cite{bl,gnn} that $\omega$ has the following asymptotic behavior at infinity, 
\begin{equation}\label{asymptcondition}
\lim_{|x|\to\infty}|x|^\frac{N-1}{2}\e^{\sqrt{V_\infty}|x|}\omega(x)=a_N>0\; \mbox{ and }\; \lim_{|x|\to\infty}|x|^\frac{N-1}{2}\e^{\sqrt{V_\infty}|x|}|\nabla\omega(x)|=b_N>0.
\end{equation}
As a consequence there are positive constants $C_1,C_2$ such that
\begin{equation}\label{inequalityasymp}
C_1\min\left\{1,|x|^{-\frac{N-1}{2}}\right\}\mathrm{e}^{-\sqrt{V_\infty}|x|} \leq \omega(x)\leq C_2\min\left\{1,|x|^{-\frac{N-1}{2}}\right\}\mathrm{e}^{-\sqrt{V_\infty}|x|} \quad \text{for all \ }x\in \mathbb{R}^{N}\smallsetminus\{0\}.
\end{equation}
We begin with some lemmas. Hereafter $C$ will denote a positive constant, not necessarily the same one.

\begin{lemma} \label{lem:exp} 
If $\mu_2>\mu_1\geq0$, there exists $C>0$ such that, for all $x_{1},x_{2}\in\rn$,
\begin{align}\label{eq:exp1}
\irn \e^{-\mu_{1}|x-x_{1}|}\e^{-\mu_{2}|x-x_{2}|}\d x &\leq C\e^{-\mu_{1}|x_{1}-x_{2}|}.
\end{align}
\end{lemma}

\begin{proof}
As $\mu_{1}|x_{1}-x_{2}|+(\mu_{2}-\mu_{1})|x-x_{2}|\leq\mu_{1}\left(  |x-x_{1}|+|x-x_{2}|\right)+(\mu_{2}-\mu_{1})|x-x_{2}| =\mu_{1}|x-x_{1}|+\mu_{2}|x-x_{2}|$, we have that
$$\irn\e^{-\mu_{1}|x-x_{1}|}\e^{-\mu_{2}|x-x_{2}|}\d x\leq\irn\e^{-\mu_{1}|x_{1}-x_{2}|}\e^{-(\mu_{2}-\mu_{1})|x-x_{2}|}\d x=\left(\irn\e^{-(\mu_{2}-\mu_{1})|x-x_{2}|}\d x\right)\e^{-\mu_{1}|x_{1}-x_{2}|},$$
as claimed.
\end{proof}

\begin{lemma}\label{lem:bl}
Let $f\in \cC^0(\rn)\cap L^{\infty}(\rn)$ and $h\in\cC^0(\rn)$ be radially symmetric functions satisfying 
$$\lim_{|x|\to\infty}f(x)|x|^a\e^{b|x|}=\eta \qquad \mbox{and}\qquad \irn|h(x)|(1+|x|^a)\e^{b|x|}\d x<\infty$$
for $a,b\geq 0$ and $\eta\in\r$. Then 
$$\lim_{|\xi|\to\infty}\left(\irn f(x-\xi)h(x)\d x\right)|\xi|^a\e^{b|\xi|}=\eta\irn h(x)\e^{-bx_1}\d x.$$
\end{lemma}

\begin{proof}
See \cite[Proposition 1.2]{bl}.
\end{proof}

\begin{lemma}\label{lemaintegrals}
The following statements hold true:
\begin{align}
&\lim_{|\xi|\to\infty}\left(\irn\omega^{2p-1}(x)\omega(x-\xi)\d x\right)|\xi|^\frac{N-1}{2}\e^{\sqrt{V_\infty}|\xi|}=\what a>0, \label{asympintegral1} \\
&\lim_{|\xi|\to\infty}\left(\irn\omega^p(x)\omega^p(x-\xi)\d x\right)|\xi|^\frac{N-1}{2}\e^{\sqrt{V_\infty}|\xi|}=0, \label{asymptintegral3} \\
&\lim_{|\xi|\to\infty}\left(\irn (V-V_\infty)^+(x)\omega^2(x-\xi)\d x\right)|\xi|^\frac{N-1}{2}\e^{2\sin\frac{\pi}{m}\sqrt{V_\infty}|\xi|}=0,\label{eq:V}
\end{align}
where $(V-V_\infty)^+:=\max\{0,V-V_\infty\}$.
\end{lemma}

\begin{proof}
Statement \eqref{asympintegral1} follows from Lemma \ref{lem:bl} with $f:=\omega$, $h:=\omega^{2p-1}$, $a:=\frac{N-1}{2}$, $b:=\sqrt{V_\infty}$ and $\eta:=a_N$. The assumptions of this lemma are given by \eqref{asymptcondition} and \eqref{inequalityasymp} because, as $p>1$,
$$\irn\omega^{2p-1}(x)(1+|x|^\frac{N-1}{2})\e^{\sqrt{V_\infty}|x|}\d x\leq C_2^{2p-1}\irn(1+|x|^\frac{N-1}{2})\e^{-(2p-2)\sqrt{V_\infty}|x|}\d x<\infty.$$
Therefore,
$$\lim_{|\xi|\to\infty}\left(\irn\omega^{2p-1}(x)\omega(x-\xi)\d x\right)|\xi|^\frac{N-1}{2}\e^{\sqrt{V_\infty}|\xi|} =a_N\irn\omega^{2p-1}(x)\mathrm{e}^{-\sqrt{V_\infty}x_1}\d x =:\what a.$$
To prove statement \eqref{asymptintegral3} fix $1<\bar p<p$. It follows from \eqref{inequalityasymp} and \eqref{eq:exp1} that
\begin{align*}
\irn\omega^p(x)\omega^p(x-\xi)\d x \leq C_2^{2p}\irn\e^{-p\sqrt{V_\infty}|x|}\e^{-p\sqrt{V_\infty}|x-\xi|}\d x \leq C\e^{-\bar p\sqrt{V_\infty}|\xi|}.
\end{align*}
Therefore,
$$\lim_{|\xi|\to\infty}\left(\irn\omega^p(x)\omega^p(x-\xi)\d x\right)|\xi|^\frac{N-1}{2}\e^{\sqrt{V_\infty}|\xi|} =C\lim_{|\xi|\to\infty}|x|^\frac{N-1}{2}\e^{(1-\bar p)\sqrt{V_\infty}|\xi|}=0.$$
Finally, as assumption $(V_3^m)$ requires that $\kappa\in(2\sin\frac{\pi}{m},2)$, tha inequality \eqref{eq:exp1} yields
\begin{align*}
0\leq\irn (V-V_\infty)^+(x)\omega^2(x-\xi)\d x \leq\irn\e^{-\kappa\sqrt{V_\infty}|x|}\e^{-2\sqrt{V_\infty}|x-\xi|}\d x\leq C\e^{-\kappa\sqrt{V_\infty}|\xi|},
\end{align*}
and statement \eqref{eq:V} follows.
\end{proof}

\begin{lemma} \label{lem:q}
If $q\geq 2$ and $a_1,\ldots,a_m\geq 0$ then
\begin{equation}\label{eq:q}
\left|\sum_{i=1}^m a_i\right|^q\geq \sum_{i=1}^m a_i^q+(q-1)\sum_{\substack{i,j=1 \\ i\neq j}}^m a_i^{q-1}a_j.
\end{equation}
\end{lemma}

\begin{proof}
See \cite[Lemma 4]{cc}
\end{proof}

\begin{proposition} \label{prop:existence}
If $m>2\ell$, then $c_V^G<\ell m\mathfrak{c}_\infty$.
\end{proposition}

\begin{proof}
Set $\xi_j:=(\e^{2\pi\mathrm{i}j/m},0,0)\in\cc\times\cc\times\r^{N-4}$, $j\in\z/m\z$, and for $R>0$ define
$$\omega_{i,j,R}(x):=\omega(x-R\vr_\ell^{-i}\xi_j),\qquad i\in\z/\ell\z, \ j\in\z/m\z,$$
and set
\begin{align}\label{Ansatz}
u_{i,R}:=\sum_{j=0}^{m-1}\omega_{i-1,j,R},\qquad i=1,\ldots,\ell.    
\end{align}
Since $\omega$ is radial, we have that $u_{i+1,R}=u_{i,R}\circ\vr_\ell$. Therefore, $\bf u_R=(u_{1,R},\ldots,u_{\ell,R})\in\mathscr{H}^G$. Furthermore, for $R$ sufficiently large, $\bf u_R$ satisfies \eqref{eq:proj_nehari}. Hence, there exists $s_R\in (0,\infty)$ such that $s_R\bf u_R\in \mathcal{N}_V^G$ and by \eqref{eq:functionalJs},
\begin{equation}\label{eq:*}
c_V^G\leq\mathcal{J}_V(s_R\bf u_R) = \dfrac{p-1}{2p} \left(\frac{\displaystyle\sum\limits_{i=1}^\ell\|u_{i,R}\|_V^2}{\Big(\displaystyle\sum\limits_{i=1}^\ell|u_{i,R}|_{2p}^{2p}+\displaystyle\sum\limits_{\substack{i,n=1 \\ n\neq i}}^\ell\beta\displaystyle\irn|u_{i,R}|^p|u_{n,R}|^p\Big)^\frac{1}{p}}\right)^\frac{p}{p-1}.
\end{equation} 
Our goal is to show that the right-hand side is smaller than $\ell m\mathfrak{c}_\infty$ for sufficiently large $R$. Observe that 
\begin{align}
|\xi_j-\xi_k|&\geq |\xi_1-\xi_0|=2\sin\frac{\pi}{m}\qquad\text{for \ }j,k\in\z/m\z, \ j\neq k, \label{xi1}\\ 
|\vr_\ell^{i-1}\xi_j-\vr_\ell^{n-1}\xi_k|&\geq |\vr_\ell\xi_0-\xi_0|=2\sin\frac{\pi}{2\ell}>2\sin\frac{\pi}{m}\qquad\text{for \ }i,n\in\{1,\ldots,\ell\}, \ i\neq n.\label{xi2}
\end{align}
Note also that, as $\omega_{0,j,R}$ solves \eqref{eq:limit_problem},
$$\langle\omega_{0,j,R},\omega_{0,k,R}\rangle_{V_\infty}=\irn\omega_{0,j,R}^{2p-1}\omega_{0,k,R}=\irn\omega^{2p-1}(x)\omega(x-R(\xi_j-\xi_k))\d x.$$
Set
\begin{align}\label{def:epsR}
	\eps_R:=\sum_{\substack{j,k=0 \\ j\neq k}}^{m-1}\langle\omega_{0,j,R},\omega_{0,k,R}\rangle_{V_\infty}=\sum_{\substack{j,k=0 \\ j\neq k}}^{m-1}\irn\omega_{0,j,R}^{2p-1}\omega_{0,k,R}.
\end{align}
It follows from \eqref{xi1} and \eqref{asympintegral1} that
\begin{equation}\label{eq:0}
\lim_{R\to\infty}\eps_R \, R^\frac{N-1}{2}\e^{2\sin\frac{\pi}{m}\sqrt{V_\infty}R}=\bar{a}>0.
\end{equation}
Now, using \eqref{eq:V} we get
\begin{align}\label{eq:1}
\|u_{1,R}\|_V^2&=\|u_{1,R}\|_{V_\infty}^2+\irn (V-V_\infty)u_{1,R}^2 \\
&\leq \sum_{j=0}^{m-1}\|\omega_{0,j,R}\|_{V_\infty}^2 + \sum_{\substack{j,k=0 \\ j\neq k}}^{m-1}\langle\omega_{0,j,R},\omega_{0,k,R}\rangle_{V_\infty} + m\sum_{j=0}^{m-1}\irn(V-V_\infty)^+\omega_{0,j,R}^2 \nonumber  \\
&=m\|\omega\|_{V_\infty}^2 + \eps_R + o(\eps_R).\nonumber 
\end{align}
From \eqref{eq:q} we derive
\begin{align} \label{eq:2}
|u_{1,R}|_{2p}^{2p}&=\irn\Big|\sum_{j=0}^{m-1}\omega_{0,j,R}\Big|^{2p} \geq \sum_{j=0}^{m-1}\irn|\omega_{0,j,R}|^{2p} + (2p-1)\sum_{\substack{j,k=0 \\ j\neq k}}^{m-1}\irn\omega_{0,j,R}^{2p-1}\omega_{0,k,R} \\
&= m\|\omega\|_{V_\infty}^2 + (2p-1)\eps_R.\nonumber
\end{align}
On the other hand, it follows from \eqref{asymptintegral3} that
\begin{align*}
0&=\lim_{R\to\infty}\left(\irn \omega_{i-1,j,R}^p\omega_{n-1,k,R}^p\right)|R\vr_\ell^{i-1}\xi_j-R\vr_\ell^{n-1}\xi_k|^\frac{N-1}{2}\e^{\sqrt{V_\infty}|R\vr_\ell^{i-1}\xi_j-R\vr_\ell^{n-1}\xi_k|} \\
&=|\vr_\ell^{i-1}\xi_j-\vr_\ell^{n-1}\xi_k|^\frac{N-1}{2}\lim_{R\to\infty}\left(\irn \omega_{i-1,j,R}^p\omega_{n-1,k,R}^p\right)R^\frac{N-1}{2}\e^{|\vr_\ell^{i-1}\xi_j-\vr_\ell^{n-1}\xi_k|\sqrt{V_\infty}R}.
\end{align*}
So, if $i\neq n$, using \eqref{xi2} and \eqref{eq:0} we obtain
$$\irn \omega_{i-1,j,R}^p\omega_{n-1,k,R}^p=o(\eps_R).$$
As a consequence,
\begin{align}\label{eq:3}
\irn|u_{i,R}|^p|u_{n,R}|^p\leq m^{2p}\irn\Big(\sum_{j=0}^{m-1}\omega_{i-1,j,R}^p\Big)\Big(\sum_{k=0}^{m-1}\omega_{n-1,k,R}^p\Big) = o(\eps_R).
\end{align}
Finally, noting that $\|u_{i,R}\|_V=\|u_{1,R}\|_V$ and $|u_{i,R}|_{2p}=|u_{1,R}|_{2p}$ and using \eqref{eq:1}, \eqref{eq:2} and \eqref{eq:3} we estimate \eqref{eq:*} as 
\begin{align*}
c_V^G&\leq\mathcal{J}_V(s_R\bf u_R)
\leq \dfrac{p-1}{2p}\left(\frac{\ell m\|\omega\|_{V_\infty}^2 + \ell\eps_R + o(\eps_R)}{\Big(\ell m\|\omega\|_{V_\infty}^2 + (2p-1)\ell\eps_R + o(\eps_R)\Big)^\frac{1}{p}}\right)^\frac{p}{p-1}\\
&\leq \frac{p-1}{2p}\left((\ell m\|\omega\|_{V_\infty}^2)^\frac{p-1}{p}-C\eps_R\right)^\frac{p}{p-1}
< \ell m \mathfrak{c}_\infty
\end{align*}
for some positive constant $C$ and sufficiently large $R$. Therefore, $c_V^G<\ell m\mathfrak{c}_\infty$, as claimed.
\end{proof}

\begin{proof}[Proof of Theorem \ref{thm:existence}]
$(i)$ follows immediately from Corollary \ref{cor:compactness} and Proposition \ref{prop:existence}, and $(ii)$ is a consequence of Corollary \ref{cor:compactness2} and Proposition \ref{prop:existence}.
\end{proof}

\section{The limit profile of the solutions under weak interaction}
\label{sec:limitprofileweak}
In this section we describe the profile of the solutions given by Theorem \ref{thm:existence} as $\beta\to 0$.

Given a sequence $(\beta_k)$ with $\beta_k<0$ for all $k\in \mathbb{N}$, we write $\mathcal{J}_{V,k}$ and $\mathcal{N}_{V,k}^G$ for the functional and the Nehari set of the system \eqref{eq:system} with $\beta=\beta_k$. We define 
\begin{align*}
	c_{V,k}^G:= \inf_{\bf u\in \mathcal{N}_{V,k}^G}\mathcal{J}_{V,k}(\bf u).
\end{align*}

\begin{proof}[Proof of Theorem \ref{thm:betatozero1}]
As $\beta_k<0$, from Proposition \ref{prop:nehari}$(b)$ we see that $0<a_0\leq\|u_{k,i}\|_V^2\leq |u_{k,i}|_{2p}^{2p}$ and from Proposition \ref{prop:nehari}$(e)$, using Hölder's and Sobolev's inequalities, we get that
$$\irn|u_{k,i}|^p|u_{k,j}|^p\leq |u_{k,i}|^p_{2p}|u_{k,j}|^p_{2p}\leq C\|u_{k,i}\|^p_V\|u_{k,j}\|^p_V\leq a_1$$
for every $k\in\n$ and $i,j=1,\ldots,\ell$. As $\beta_k\to 0$, passing to a subsequence, we may assume that $\beta_1\leq\beta_k$ and that $|\beta_k| a_1<\frac{a_0}{2}$ for all $k$. Therefore,  
$$|u_{1,i}|_{2p}^{2p}+\beta_k\irn|u_{1,i}|^p|u_{1,j}|^p>0\qquad\text{for any \ }k\in\n.$$
Then, by Proposition \ref{prop:nehari}$(d)$, there exists $s_k\in(0,\infty)$ such that $s_k\bf u_1\in\cN_{V,k}^G$ and 
\begin{align*}
\cJ_{V,k}(\bf u_k)\leq \cJ_{V,k}(s_k\bf u_1)&= \frac{p-1}{2p} \left(\frac{\sum\limits_{i=1}^\ell\|u_{1,i}\|_V^2}{\Big(\sum\limits_{i=1}^\ell|u_{1,i}|_{2p}^{2p}+\sum\limits_{\substack{i,j=1 \\ j\neq i}}^\ell\beta_k\irn|u_{1,i}|^p|u_{1,j}|^p\Big)^\frac{1}{p}}\right)^\frac{p}{p-1} \\
&\leq\frac{p-1}{2p} \left(\frac{\sum\limits_{i=1}^\ell\|u_{1,i}\|_V^2}{\Big(\sum\limits_{i=1}^\ell|u_{1,i}|_{2p}^{2p}+\sum\limits_{\substack{i,j=1 \\ j\neq i}}^\ell\beta_1\irn|u_{1,i}|^p|u_{1,j}|^p\Big)^\frac{1}{p}}\right)^\frac{p}{p-1}=\cJ_{V,1}(\bf u_1).
\end{align*}
From Proposition \ref{prop:existence} we derive
\begin{equation}\label{eq:inequality0}
c_{V,k}^G\leq c_{V,1}^G <\ell m\mathfrak{c}_\infty\qquad\text{for every \ }k\in\n.
\end{equation} 

Following the proof of Theorem \ref{thm:splitting}, after passing to a subsequence, there exist $\xi_k\in\rn$ and $C_0>0$ such that, either $\xi_k=0$, or $\xi_k=(z_k,0)$ with $|z_k|\to\infty$, and in both cases the function
$$w_{k,1}(x):=u_{k,1}(x+\xi_k)$$
satisfies that $w_{k,1}$ is $O(N-4)$-invariant and
\begin{equation} \label{eq:nonnull2}
\int_{B_{C_0+1}(0)}|w_{k,1}|^{2p}=\int_{B_{C_0+1}(\xi_k)}|u_{k,1}|^{2p}>\delta>0\qquad\forall k\in\n.
\end{equation}
Since the sequence $(w_{k,1})$ is bounded in $H^1(\rn)$, passing to a subsequence we get that $w_{k,1}\rh w_1$ weakly in $H^1(\rn)$, $w_{k,1}\to w_1$ in $L^{2p}_\mathrm{loc}(\rn)$ and $w_{k,1}\to w_1$ a.e. in $\rn$. Hence, $w_1$ is $O(N-4)$-invariant and $w_1\geq 0$. Furthermore, from \eqref{eq:nonnull2} we see that  $w_1\neq 0$. 

Assume that $\xi_k=(z_k,0)$ with $|z_k|\to\infty$. Set $V_k(x):=V(x+\xi_k)$. Following the proof of \eqref{eq:norms} we obtain
\begin{equation} \label{eq:norms2}
\|u_{k,1}\|_V^2=\Big\|u_{k,1}-\sum_{i=0}^{m-1}(w_1\circ\vartheta_m^{-i})(\,\cdot\,-\vartheta_m^i\xi_k)\Big\|_V^2+m\|w_1\|_{V_\infty}^2+o(1).
\end{equation}
Define $\what w_1(x):=w_1(x-\xi_{k})$ and $w_{k,j}(x):=u_{k,j}(x+\xi_k)$. Then, performing a change of variable we derive
\begin{align} \label{eq:solution}
0&=\partial_1\mathcal{J}_{V,k}(\bf u_k)\what w_1 \\
&=\irn(\nabla w_{k,1}\cdot\nabla w_1+V_kw_{k,1}w_1)-\irn|w_{k,1}|^{2p-2}w_{k,1}w_1 -\beta_k\sum_{j=2}^\ell\irn|w_{k,j}|^p|w_{k,1}|^{p-2}w_{k,1}w_1 \nonumber
\end{align}
and, as $\beta_k\to 0$, passing to the limit and using $(V_2)$ we get that $\|w_1\|_{V_\infty}^2=|w_1|_{2p}^{2p}$. This shows that $w_1\in\cM_\infty$, and from \eqref{eq:norms2} and \eqref{eq:inequality0} we obtain
\begin{align*}
m\mathfrak{c}_\infty\leq\frac{p-1}{2p}m\|w_1\|_{V_\infty}^2\leq\liminf_{k\to\infty}\frac{p-1}{2p}\|u_{k,1}\|_V^2=\liminf_{k\to\infty}\frac{1}{\ell}c_{V,k}^G\leq\frac{1}{\ell}c_{V,1}^G< m\mathfrak{c}_\infty.
\end{align*}
This is a contradiction. 

As a consequence, $\xi_k=0$. Then, $w_{k,j}=u_{k,j}$ and $V_k=V$. Set $u_{0,j}:=w_j\neq 0$. Then, $u_{0,1}, \ldots , u_{0,\ell}$ satisfy $(S_1^G)$ and $(S_2)$ and, passing to the limit, from \eqref{eq:solution} we get that $\|u_{0,1}\|^2_V=|u_{0,1}|_{2p}^{2p}$.  Hence, $u_{0,1}\in\cM^G$ and
\begin{equation}\label{in:lim}
\mathfrak{c}^G \leq \frac{p-1}{2p}\|u_{0,1}\|_V^2\leq \liminf_{k\to\infty}\frac{p-1}{2p}\|u_{k,1}\|_V^2=\liminf_{k\to\infty}\frac{1}{\ell}c_{V,k}^G.  
\end{equation}
We claim that $\ell \mathfrak{c}^G=\lim_{k\to\infty}c_{V,k}^G$. To prove this claim let $v_k \in \mathcal{M}^G$ be such that  $J(v_k) \to \mathfrak{c}^G$. Set $v_{k,1} := v_k$ and let $v_{k, j+1}$ satisfy the pinwheel condition $(S_2)$ for every $j = 1,\ldots , \ell-1$. Set ${\bf v_k = (v_{k,1},\ldots,v_{k,\ell})}$. Since $(v_k)$ is bounded in $H^1(\rn)$ and $\beta_k \to 0$, we have that
\begin{align*}
\lim_{k\to\infty} \beta_k\irn|v_{k,j}|^p|v_{k,i}|^{p} = 0 \qquad  \mbox{for every }  i, j.
\end{align*}
Thus, by Proposition \ref{prop:nehari}(d), for each sufficiently large $k$  there exists $s_k \in (0,\infty)$ such that $s_k\bf{v}_k \in \mathcal{N}^G_{V,k}$ and $s_k \to 1$. Therefore,
\begin{align*}
c_{V,k}^G\leq\mathcal{J}_{V,k}(s_k\bf{v}_k)=\frac{p-1}{2p}\sum_{j=1}^\ell\|s_k v_{k,j}\|_V^2 = \frac{p-1}{2p}\ell s_k^2\|v_{k}\|_V^2=\ell s_k^2 J(v_k).
\end{align*}
It follows that $\limsup_{k\to\infty}c_{V,k}^G\leq \ell \mathfrak{c}^G$ which, together with \eqref{in:lim}, yields $\ell \mathfrak{c}^G=\lim_{k\to\infty}c_{V,k}^G$. Hence, the inequalities in \eqref{in:lim} are, in fact, equalities. As a consequence, $u_{k,j}\to u_{0,j}$ strongly in $ H^1(\rn)$  and from \eqref{eq:inequality0} we get that
\begin{align*}  
\mathfrak{c}^G=\frac{p-1}{2p}\|u_{0,j}\|_V^2 = \lim_{k\to\infty} \frac{ {c}_{V,k}^G}{\ell}  \leq \frac{ {c}_{V,1}^G}{\ell} <  m \mathfrak{c}_\infty,
\end{align*}	
for every $j=1,\ldots, \ell$, as claimed.
\end{proof}

\begin{proof}[Proof of Theorem \ref{thm:betatozero2}]
Arguing as in the proof of Theorem \ref{thm:splitting2} we see that, after passing to a subsequence, there exist $\widetilde y_k\in \r^{N-4}$ such that, if we set $\widetilde u_{k,j}(z,y):=v_{k,j}(z,y-\widetilde y_k)$, then $\widetilde{\bf u}_k:=(\widetilde u_{k,1},\ldots,\widetilde u_{k,\ell})\in\cN^G_{V_\infty,k}$, $\cJ_{V_\infty,k}(\widetilde{\bf u}_k)\to c_{V_\infty,k}^G$ and there exist $\xi_k\in\rn$ and $C_0>0$ such that, either $\xi_k=0$, or $\xi_k=(z_k,0)$ with $|z_k|\to\infty$, and in both cases
$$w_{k,1}(x):=\widetilde u_{k,1}(x+\xi_k)$$
satisfies
\begin{equation*}
\int_{B_{C_0+1}(0)}|w_{k,1}|^{2p}=\int_{B_{C_0+1}(\xi_k)}|u_{k,1}|^{2p}>\delta>0\qquad\forall k\in\n.
\end{equation*}
Now we follow the proof of Theorem \ref{thm:betatozero1} to show that $\widetilde u_{k,1}\to\widetilde u_{0,1}$ strongly in $H^1(\rn)$, where $\widetilde u_{0,1}$ is a positive least energy solution to \eqref{eq:limit_problem}.

Since,  up to translations, $\omega$ is the only positive solution to \eqref{eq:limit_problem} and $\widetilde u_{0,1}$ is $G$-invariant, there exists $\widetilde y\in\r^{N-4}$ such that $\omega(z,y)= \widetilde{u}_{0,1}(z,y+\widetilde y)$ for every $(z,y)\in\rn$. So, setting $y_k := \widetilde y_k - \widetilde y$ and ${u_{k,j}(z,y):= v_{k,j}(z, y - y_k) }$, we get that $u_{k,j}\to\omega$ strongly in $H^1(\rn)$ for every $j=1,\ldots,\ell$, as claimed.
\end{proof}

\section{The limit profile of the solutions under strong interaction}
\label{sec:limitprofilestrong}
In this section we describe the profile of the solutions given by Theorem \ref{thm:existence} as $\beta\to -\infty$. The set of {\it weak $(G,V)$-pinwheel partitions} is defined as
\begin{align*}
  \mathcal{W}_{V}^G:= \{ \bf u \in \mathscr{H}^G: u_i \neq 0, \|u_i\|_V^2 = |u_i|_{2p}^{2p} \ \mbox{ and } \ u_iu_j \equiv 0 \ \mbox{if} \ i\neq j\}.
\end{align*}
Set 
\begin{align*}
	\widehat{c}_V^G:=  \inf_{ \bf u \in \mathcal{W}_V^G} \frac{p-1}{2p}\sum_{i=1}^\ell\|u_i\|^2_V=\inf_{ \bf u \in \mathcal{W}_V^G} \frac{p-1}{2p}\ell\,\|u_1\|^2_V.
\end{align*}
Note that $\mathcal{W}_V^G\subset \mathcal{N}_V^G$ for every $\beta<0$, hence 
\begin{align}\label{boundC}
c_V^G\leq\widehat{c}_V^G.	
\end{align}
The following lemma will play a crucial role.

\begin{lemma}\label{lem:upperbound}
If $m$ is even and $2\sin\frac{\pi}{m}<\sin\frac{\pi}{2\ell}$, then $\widehat{c}_V^G <\ell m\mathfrak{c}_\infty.$ 
\end{lemma} 

\begin{proof}
Choose $\delta\in \left(4\sin\frac{\pi}{m},2\sin\frac{\pi}{2\ell}\right)$,  $\eps \in \big(0, \tfrac{2\sin\frac{\pi}{2\ell} - \delta}{2\sin\frac{\pi}{2\ell} + \delta}\big)$ and a radial cut-off function $\chi \in\cC^\infty_c(\mathbb{R}^N)$ satisfying $0\leq\chi\leq 1$, $\chi(x) = 1$ for $|x| \leq {1} - \eps$ and $\chi(x) = 0$ for $|x| \geq 1$. For $s>0$ set $\chi_s(x):=\chi\left(\frac{x}{s}\right)$ and $\omega_s:= \chi_s\omega$, where $\omega$ is the positive radial least energy solution to \eqref{eq:limit_problem}. In \cite[Lemma 2]{cw} (see also \cite[Lemma 4.1]{cp}) it is shown that
\begin{equation} \label{eq:est1}
\left|\|\omega\|_{V_\infty}^2 - \|\omega_s\|_{V_\infty}^2\right| = O(\e^{-2(1-\eps)\sqrt{V_\infty}s}) \qquad \mbox{and} \qquad \big||\omega|_{2p}^{2p}-|\omega_{s}|_{2p}^{2p}\big| = O(\e^{-2p(1-\eps)\sqrt{V_\infty}s}).
\end{equation}
Let ${\xi_j:=(\e^{2\pi\mathrm{i}j/m},0,0)\in\cc\times\cc\times\r^{N-4}}$ and set
 \begin{align}\label{auxlemma51}
 	2\sin\tfrac{\pi}{m}<\frac{\delta}{2}<r := \frac{2\sin\tfrac{\pi}{2\ell} + \delta}{4}<\sin\tfrac{\pi}{2\ell}.
\end{align} 
For each $i\in\z/\ell\z$, $j\in\z/m\z$ and $R>1$ define
 \begin{align*}
 	\omega_{i,j,R}(x):=\omega_{rR}(x-R\vr_\ell^{-i}\xi_j)= \chi_{rR}(x-R\vr_\ell^{-i}\xi_j)\omega(x-R\vr_\ell^{-i}\xi_j)
 \end{align*}
and let
\begin{align*}
\widehat u_{i,R}:=\sum_{j=0}^{m-1}\omega_{i-1,j,R},  \quad i=1,\ldots,\ell.
\end{align*}
Then $\widehat{\bf u}_R:=(\widehat{u}_{1,R},\ldots,\widehat{u}_{\ell,R})\in\mathscr{H}^G$ and, by \eqref{xi1}, \eqref{xi2} and \eqref{auxlemma51}, $\supp (\widehat{u}_{i,R})\cap \supp (\widehat{u}_{j,R}) = \emptyset$ if $i\neq j$. So, setting $s_R:=\Big(\frac{\|s_R\widehat{u}_{1,R}\|_V^2}{|s_R\widehat{u}_{1,R}|_{2p}^{2p}}\Big)^\frac{1}{2p-2}$, we have that $s_R\widehat{\bf u}_R \in \mathcal{W}_V^G$ and
\begin{equation}\label{eq:what c}
\what c_V^G\leq\frac{p-1}{2p}\ell\,\|s_R\widehat u_{1,R}\|^2_V=\frac{p-1}{2p}\ell\left(\frac{\|\widehat{u}_{1,R}\|_V^2}{(|\widehat{u}_{1,R}|_{2p}^{2p})^{\frac{1}{p}}}\right)^{\frac{p}{p-1}}.
\end{equation}
Let $\eps_R$ be defined as in \eqref{def:epsR}. Recall that 
$$\lim_{R\to\infty}\eps_R \, R^\frac{N-1}{2}\e^{2\sin\frac{\pi}{m}\sqrt{V_\infty}R}=\bar{a}>0.$$
Since $r(1-\eps)\geq\frac{\delta}{2}>2\sin\frac{\pi}{m}$, from \eqref{eq:est1} we get
\begin{equation}\label{eq:est2}
\|\omega \|_{V_{\infty}}^2=\|\omega_{rR}\|_{V_{\infty}}^2 + o(\eps_R)\qquad\text{and}\qquad |\omega|_{2p}^{2p}=\|\omega_{rR}\|_{2p}^{2p} + o(\eps_R)\qquad\text{as \ }R\to\infty.
\end{equation}
Then, noting that $\omega_{rR}\leq\omega$, from \eqref{eq:V} we derive
\begin{align} \label{eq:A}
\|\widehat{u}_{1,R}\|_V^2 &=\|\widehat{u}_{1,R}\|_{V_{\infty}}^2+\int_{\rn}(V-V_{\infty})\widehat{u}^2_{1,R}\\
&\leq m\|\omega_{rR}\|_{V_{\infty}}^2+\sum_{\substack{j,k=0 \\ j\neq k}}^{m-1}\langle \omega_{rR}(\cdot-R\xi_j), \omega_{rR}(\cdot-R\xi_k)\rangle_{V_{\infty}} +m\sum_{j=0}^{m-1}\int_{\rn}(V-V_{\infty})^+\omega_{rR}^2(\cdot-R\xi_j)\nonumber \\
&=m\|\omega \|_{V_{\infty}}^2 +\sum_{\substack{j,k=0 \\ j\neq k}}^{m-1}\langle \omega_{rR}(\cdot-R\xi_j), \omega_{rR}(\cdot-R\xi_k)\rangle_{V_{\infty}}  + o(\eps_R),\nonumber
\end{align}
and using Lemma \ref{lem:q} we obtain
\begin{align} \label{eq:B}
|\widehat{u}_{1,R}|_{2p}^{2p} &=\int_{\rn} \Big|\sum_{j=0}^{m-1}\omega_{rR}(\cdot-R\xi_j)\Big|^{2p} \\
&\geq m|\omega_{rR}|_{2p}^{2p}+(2p-1)\sum_{\substack{j,k=0 \\ k\neq j}}^{m-1}\int_{\rn}\omega_{rR}^{2p-1}(\cdot-R\xi_j)\omega_{rR}(\cdot-R\xi_k)\nonumber \\
&=m|\omega|_{2p}^{2p}+(2p-1)\sum_{\substack{j,k=0 \\ k\neq j}}^{m-1}\int_{\rn}\omega_{rR}^{2p-1}(\cdot-R\xi_j)\omega_{rR}(\cdot-R\xi_k)+o(\eps_R).\nonumber
\end{align}
For $k\neq j$ set $\zeta_{jk}=\xi_j-\xi_k$. Then, as $0\leq\chi\leq 1$, we derive from \eqref{inequalityasymp} that
\begin{align} \label{eq:interaction 1}
&\irn |\omega^{2p-1}\omega(\cdot - R\zeta_{kj})-\omega^{2p-1}\omega_{rR}(\cdot - R\zeta_{kj})|=\irn |\omega^{2p-1}(\omega(\cdot - R\zeta_{kj})-\omega_{rR}\left(\cdot - R\zeta_{kj})\right)|\\
&\leq \int_{|x-R\zeta_{kj}|>r(1-\eps)R} \omega^{2p-1}(x)\omega(x - R\zeta_{kj})\d x\leq C\int_{|x-R\zeta_{kj}|>r(1-\eps)R}\omega^{2p-1}(x)\e^{-\sqrt{V_\infty}|x-R\zeta_{kj}|}\d x \nonumber \\
&\leq C\e^{-r(1-\eps)\sqrt{V_\infty}R}|\omega|_{2p-1}^{2p-1}=o(\eps_R), \nonumber
\end{align}
because $r(1-\eps)>2\sin\frac{\pi}{m}$. Similarly,
\begin{align} \label{eq:interaction 2}
&\irn |\omega^{2p-1}\omega_{rR}(\cdot - R\zeta_{kj})-\omega_{rR}^{2p-1}\omega_{rR}(\cdot - R\zeta_{kj})|=\irn |(\omega^{2p-1}-\omega_{rR}^{2p-1})\omega_{rR}(\cdot - R\zeta_{kj})|\\
&\leq \int_{|x|>r(1-\eps)R} \omega^{2p-1}(x)\omega_{rR}(x - R\zeta_{kj})\d x\leq C\int_{|x|>r(1-\eps)R}\e^{-(2p-1)\sqrt{V_\infty}|x|}\omega(\cdot - R\zeta_{kj})\d x \nonumber \\
&\leq C\e^{-r(1-\eps)\sqrt{V_\infty}R|}|\omega|_1=o(\eps_R). \nonumber
\end{align}
From \eqref{def:epsR}, \eqref{eq:interaction 1} and \eqref{eq:interaction 2} we obtain
\begin{equation}\label{eq:interaction 3}
\int_{\rn}\omega_{rR}^{2p-1}(\cdot-R\xi_j)\omega_{rR}(\cdot-R\xi_k)=\int_{\rn}\omega_{rR}^{2p-1}\omega_{rR}(\cdot-R\zeta_{kj})=\eps_R + o(\eps_R)
\end{equation}
and, as $\omega$ solves \eqref{eq:limit_problem}, we derive from \eqref{def:epsR} and \eqref{eq:interaction 1} that
\begin{equation} \label{eq:interaction 4}
\langle \omega, \omega_{rR}(\cdot-R\zeta_{kj})\rangle_{V_{\infty}}=\irn\omega^{2p-1}\omega_{rR}(\cdot-R\zeta_{kj})=\eps_R + o(\eps_R).
\end{equation}
Note that \eqref{eq:interaction 2} holds true also for $p=1$. Therefore,
\begin{align} \label{eq:interaction 5}
\langle \omega_{rR}-\omega, \omega_{rR}(\cdot-R\zeta_{kj})\rangle_{V_{\infty}}&=\irn\nabla(\omega_{rR}-\omega)\cdot\nabla\omega_{rR}(\cdot-R\zeta_{kj}) + V_\infty\irn(\omega_{rR}-\omega)\omega_{rR}(\cdot-R\zeta_{kj}) \\
&=\irn\nabla(\omega_{rR}-\omega)\cdot\nabla\omega_{rR}(\cdot-R\zeta_{kj}) + o(\eps_R). \nonumber
\end{align}
Using the asymptotic \eqref{asymptcondition} we get
\begin{align*}
&|\nabla\omega_{rR}(x)|=|\chi_{rR}(x)\nabla\omega(x) + \tfrac{1}{rR}\omega(x)\nabla\chi(x)|\leq C(|\nabla\omega(x)|+\omega(x))\leq C\e^{-\sqrt{V_\infty}|x|}, \\
&|\nabla(\omega_{rR}-\omega)(x)|=|(\chi_{rR}(x)-1)\nabla\omega(x) + \tfrac{1}{rR}\omega(x)\nabla\chi(x)|\leq C(|\nabla\omega(x)|+\omega(x))\leq C\e^{-\sqrt{V_\infty}|x|}.
\end{align*} 
Therefore,
\begin{align*}
\irn|\nabla(\omega_{rR}-\omega)\cdot\nabla\omega_{rR}(\cdot-R\zeta_{kj})|&\leq\int_{|x|>r(1-\eps)R} |\nabla(\omega_{rR}-\omega)(x)|\,|\nabla\omega_{rR}(x-R\zeta_{kj})|\d x \\
&\leq C\e^{-r(1-\eps)\sqrt{V_\infty}R}|\nabla\omega_{rR}|_1 = o(\eps_R).
\end{align*}
Combining this estimate with \eqref{eq:interaction 4} and \eqref{eq:interaction 5} we obtain
\begin{equation} \label{eq:interaction 6}
\langle \omega_{rR}(\cdot-R\xi_j), \omega_{rR}(\cdot-R\xi_k)\rangle_{V_{\infty}}=\langle \omega_{rR}, \omega_{rR}(\cdot-R\zeta_{k,j})\rangle_{V_{\infty}}=\eps_R + o(\eps_R).
\end{equation}
From \eqref{eq:A}, \eqref{eq:B}, \eqref{eq:interaction 3} and \eqref{eq:interaction 6} we derive
\begin{align*}
\frac{\|\widehat{u}_{1,R}\|_V^2}{(|\widehat{u}_{1,R}|_{2p}^{2p})^{\frac{1}{p}}}&=\frac{m\|\omega\|_{V_{\infty}}^2+\eps_R+o(\eps_R)}{\left(m|\omega|_{2p}^{2p}+(2p-1)\eps_R+o(\eps_R)\right)^{\tfrac{1}{p}}}\leq ( m\|\omega\|_{V_\infty}^2)^\frac{p-1}{p}-C\eps_R
\end{align*}
for some positive constant $C$ and sufficiently large $R$. Inserting this estimate into \eqref{eq:what c} we get that $\widehat{c}_{V}^G < \ell m \mathfrak{c}_{\infty}$, as claimed.
\end{proof}

For any $G$-invariant open subset $\o$ of $\rn$ we set $H^1_0(\o)^G:=\{u\in H^1_0(\o):u\text{ is }G\text{-invariant}\}$. Consider the problem
\begin{equation} \label{prob_omega}
\begin{cases}
-\Delta u + V(x)u = |u|^{2p-2}u, \\
u\in H^1_0(\o)^G.
\end{cases}
\end{equation}
A minimizer for the functional $J$, as defined in \eqref{functionalJsingle}, on $\mathcal{M}_{\Omega}^G:=\{ u\in H^1_0(\o)^G: u\neq 0, \ \|u\|_V^{2}=|u|_{2p}^{2p} \}$ is a least energy solution to \eqref{prob_omega}. Set
\begin{equation*}
\mathfrak{c}_{\o}^G:=\inf_{ u\in\mathcal{M}_{\Omega}^G }J(u)= \inf_{ u\in\mathcal{M}_{\Omega}^G } \frac{p-1}{2p}\|u\|_V^{2}.
\end{equation*}

\begin{definition} \label{def:partition}
An $\ell$-tuple $(\o_1, \dots, \o_{\ell})$ of subsets of $\rn$ is an \textbf{optimal $(G,\vr_\ell)$-pinwheel partition} for the problem \eqref{eq:single} if
\begin{itemize}
\item[$(P_1)$] $\o_i\neq\emptyset$, \ $\o_i\cap\o_j=\emptyset$ if $i\neq j$, \ $\rn=\bigcup_{i=1}^\ell\overline{\o}_i$ \ and \ $\o_i$ is open and $G$-invariant,
\item[$(P_2)$] for each  $i=1,\ldots,\ell$, the value $c_{\o_i}^G$ is attained by some function $u_i\in\cM_{\o_i}^G$,
\item[$(P_3)$] $\vr_\ell(\o_2)=\o_1, \ \ldots \ , \ \vr_\ell(\o_\ell)=\o_{\ell-1}, \ \ \vr_\ell(\o_1)=\o_\ell$,
\item[$(P_4)$] and
\begin{equation*}
\sum_{j=1}^{\ell}\mathfrak{c}_{\o_{j}}^G =\inf_{(\Theta_{1},\dots \Theta_{\ell}) \in \mathcal{P}_V^G}	\sum_{j=1}^{\ell}\mathfrak{c}_{\Theta_{j}}^G.
\end{equation*}
where $\cP_V^G$ is the set of all $\ell$-tuples $(\Theta_{1},\dots,\Theta_{\ell})$ of subsets of $\rn$ that satisfy $(P_1)$, $(P_2)$ and $(P_3)$.
\end{itemize}
\end{definition}

Note that, as the sets $\o_1, \dots, \o_{\ell}$ are mutually isometric and exhaust $\rn$, every set $\o_i$ of the partition is necessarily unbounded.

\begin{proof}[Proof of Theorem \ref{theorembetainfinity1}]
$(i):$ \ We write $\mathcal{J}_{V,k}$ and $\mathcal{N}_{V,k}^G$ for the functional and the Nehari set of the system \eqref{eq:system} with $\beta=\beta_k$. It follows from $\eqref{boundC}$ that
\begin{align*}
	c_{V,k}^G=\cJ_{V,k}(\bf u_k)\leq\what c_V^G\qquad\text{for all \ }k\in\n.
\end{align*}
Therefore, $(u_{k,j})$ is bounded in $H^1(\rn)$ for each $j=1,\ldots,\ell$.  Following the proof of Theorem \ref{thm:splitting}, after passing to a subsequence, there exist $\xi_k\in\rn$ and $C_0>0$ such that, either $\xi_k=0$, or $\xi_k=(z_k,0)$ with $|z_k|\to\infty$, and in both cases the function
$$w_{k,1}(x):=u_{k,1}(x+\xi_k)$$
is $O(N-4)$-invariant and satisfies
\begin{equation} \label{eq:nonnull3}
\int_{B_{C_0+1}(0)}|w_{k,1}|^{2p}=\int_{B_{C_0+1}(\xi_k)}|u_{k,1}|^{2p}>\delta>0\qquad\forall k\in\n.
\end{equation}
Since the sequence $(w_{k,1})$ is bounded in $H^1(\rn)$, passing to a subsequence we get that $w_{k,1}\rh w_1$ weakly in $H^1(\rn)$, $w_{k,1}\to w_1$ in $L^{2p}_\mathrm{loc}(\rn)$ and $w_{k,1}\to w_1$ a.e. in $\rn$. Hence, $w_1$ is $O(N-4)$-invariant and $w_1\geq 0$. From \eqref{eq:nonnull3} we see that  $w_1\neq 0$. 

Assume that $\xi_k=(z_k,0)$ with $|z_k|\to\infty$. Set $V_k(x):=V(x+\xi_k)$. Following the proof of \eqref{eq:norms} we obtain
\begin{equation} \label{eq:norms3}
\|u_{k,1}\|_V^2=\Big\|u_{k,1}-\sum_{i=0}^{m-1}(w_1\circ\vartheta_m^{-i})(\,\cdot\,-\vartheta_m^i\xi_k)\Big\|_V^2+m\|w_1\|_{V_\infty}^2+o(1).
\end{equation}
Define $\what w_1(x):=w_1(x-\xi_{k})$ and $w_{k,j}(x):=u_{k,j}(x+\xi_k)$. Performing a change of variable and recalling that $w_{k,1}\geq 0$ and $w_1\geq 0$, we see that
\begin{align} \label{eq:solution2}
0&=\partial_1\mathcal{J}_{V,k}(\bf u_k)\what w_1 \\
&=\irn(\nabla w_{k,1}\cdot\nabla w_1+V_kw_{k,1}w_1)-\irn|w_{k,1}|^{2p-2}w_{k,1}w_1 -\beta_k\sum_{j=2}^\ell\irn|w_{k,j}|^p|w_{k,1}|^{p-2}w_{k,1}w_1 \nonumber \\
&\geq \irn(\nabla w_{k,1}\cdot\nabla w_1+V_kw_{k,1}w_1)-\irn|w_{k,1}|^{2p-2}w_{k,1}w_1. \nonumber
\end{align}
Passing to the limit and using $(V_2)$ we get that $\|w_1\|_{V_\infty}^2\leq|w_1|_{2p}^{2p}$. Hence, there exists $s\in(0,1]$ such that $\|sw_1\|_{V_\infty}^2=|sw_1|_{2p}^{2p}$, i.e., $sw_1\in\cM_\infty$. From \eqref{eq:norms3}, \eqref{boundC} and Lemma \ref{lem:upperbound} we obtain
\begin{align*}
m\mathfrak{c}_\infty\leq\frac{p-1}{2p}m\|sw_1\|_{V_\infty}^2\leq\frac{p-1}{2p}m\|w_1\|_{V_\infty}^2\leq\liminf_{k\to\infty}\frac{p-1}{2p}\|u_{k,1}\|_V^2=\liminf_{k\to\infty}\frac{1}{\ell}c_{V,k}^G\leq\frac{1}{\ell}\what c_V^G< m\mathfrak{c}_\infty.
\end{align*}
This is a contradiction.

As a consequence, $\xi_k=0$. Then, $w_{k,j}=u_{k,j}$ and $V_k=V$. Set $u_{\infty,j}:=w_j\neq 0$. Then, $\bf u_\infty=(u_{\infty,1}, \ldots , u_{\infty,\ell})\in\mathscr H^G$, $u_{\infty,j}\geq 0$, $u_{\infty,j}\neq 0$, and letting $k\to\infty$ in \eqref{eq:solution2} we get that $\|u_{\infty,1}\|^2_V\leq|u_{\infty,1}|_{2p}^{2p}$. Furthermore, since $(\mathcal{J}_{V,k}^G)'(\mathbf{u}_k)=0$, for each $i\neq j$ we have that
\begin{align*}
0\leq\int_{\rn}|u_{k,j}|^p|u_{k,i}|^p\leq \dfrac{|u_{k,i}|_{2p}^{2p}}{|\beta_k|}\leq \dfrac{C}{|\beta_k|}.
\end{align*}
As $\beta_k\rightarrow -\infty$, passing to the limit and using Fatou's lemma we obtain
\begin{align*}
0\leq\int_{\rn}|u_{\infty,j}|^p|u_{\infty,i}|^p\leq \liminf_{k\to\infty} \int_{\rn}|u_{k,j}|^p|u_{k,i}|^p=0.
	 \end{align*}
Therefore, $u_{\infty,j}u_{\infty,i}=0$ a.e. in $\mathbb{R}^N$ whenever $i\neq j$. Let $s\in (0,1]$ be such that $\|su_{\infty,1}\|_V^2=|su_{\infty,1}|_{2p}^{2p}$. Then, $s\bf u_{\infty}\in \mathcal{W}_V^G$ and using \eqref{boundC} we get
  \begin{align*}
  	\widehat{c}_V^G\leq\frac{p-1}{2p}\sum_{i=1}^{\ell}\|su_{\infty,i}\|_V^2 \leq \frac{p-1}{2p}\sum_{i=1}^{\ell}\|u_{\infty,i}\|_V^2\leq \frac{p-1}{2p}\sum_{i=1}^{\ell}\liminf_{k\to\infty}\|u_{k,i}\|_V^2\leq \widehat{c}_V^G.
  \end{align*}
Thus, $s=1$, $\bf u_{\infty}\in \mathcal{W}_V^G$, $u_{k,i}\rightarrow u_{\infty,i}$ strongly in $H^1(\mathbb{R}^N)$ and 
\begin{align*}
	\widehat{c}_V^G= \dfrac{p-1}{2p}\sum_{i=1}^{\ell}\|u_{\infty,i}\|_V^2.
\end{align*}
Finally, as $\|u_{\infty,i}\|_V^2=\lim_{k\to \infty}\|u_{k,i}\|_V^2=\lim_{k\to \infty}|u_{k,i}|_{2p}^{2p}=|u_{\infty,i}|_{2p}^{2p}$, from
\begin{align*}
	\lim_{k\to \infty}\|u_{k,i}\|_V^2=\lim_{k\to \infty}|u_{k,i}|_{2p}^{2p}+\lim_{k\to \infty}\beta_k\sum_{\substack{i,j=1 \\ i\neq j}}\int_{\r^N}|u_{k,j}|^p|u_{k,i}|^p,
\end{align*}
we obtain
\begin{align*}
	\int_{\r^N}\beta_k|u_{k,j}|^p|u_{k,i}|^p\rightarrow 0 \; \mbox{as  }\, k\rightarrow \infty\quad \mbox{whenever }\, i\neq j.
\end{align*}
This completes the proof of $(i)$. The proof of statements $(ii)$ and $(iii)$ is the same as in \cite[Theorem 1.4]{cp}. We give the details for the sake of completeness.

$(ii):$ \ As in \cite[Lemma 5.1]{cp} one sees that $(u_{k,j})$ is uniformly bounded in $L^\infty(\rn)$ for every $j=1,\ldots,\ell$. Then, \cite[Theorem B2]{cpt} states that $(u_{k,i})$ is uniformly bounded in $\cC^{0,\alpha}(K)$ for every compact subset $K$ of $\rn$ and $\alpha\in (0,1)$. By the Arzela-Ascoli theorem, $u_{\infty,j}\in\cC^0(\mathbb{R}^N)$ and, hence, $\Omega_j=\{x\in \mathbb{R}^N: u_{\infty,j}(x)>0\}$ is open for each $j=1,\dots, \ell$. Clearly $(\Omega_1,\dots, \Omega_{\ell})$ satisfy $(P_1)$ and $(P_3)$. We claim that $\frac{p-1}{2p}\|u_{\infty,j}\|_V^2=\mathfrak{c}_{\Omega_j}^G$. Arguing by contradiction, assume that $\frac{p-1}{2p}\|u_{\infty,j}\|_V^2> \mathfrak{c}_{\Omega_j}^G$. Then there exists $(v_1,\dots, v_{\ell}) \in \mathcal{W}_V^G$ such that $\mathfrak{c}_{\Omega_1}^G\leq\frac{p-1}{2p}\|v_1\|_V^2<\frac{p-1}{2p}\|u_{\infty,1}\|_V^2$. As a consequence,
\begin{align*}
\frac{p-1}{2p}\sum_{j=1}^{\ell}\|v_{j}\|_V^2<\frac{p-1}{2p}\sum_{j=1}^{\ell}\|u_{\infty,j}\|_V^2=\widehat{c}_V^G,
\end{align*}
which contradicts the definition of $\widehat{c}_V^G$. This proves $(P_2)$. Furthermore, we get that
\begin{align*}
\inf_{(\Theta_1,\dots,\Theta_{\ell})\in \mathcal{P}_V^G}\sum_{j=1}^{\ell}\mathfrak{c}_{\Theta_j}^G\leq \sum_{j=1}^{\ell}\dfrac{p-1}{2p}\|u_{\infty,j}\|_V^2=\widehat{c}_V^G\leq \inf_{(\Theta_1,\dots,\Theta_{\ell})\in \mathcal{P}_V^G}\sum_{j=1}^{\ell}\mathfrak{c}_{\Theta_j}^G.	
\end{align*}
This proves $(P_4)$ and shows that $(\Omega_1,\dots, \Omega_{\ell})$ is an optimal $(G,\vr_\ell)$-pinwheel partition for the problem  \eqref{eq:single}.

$(iii):$ \ As shown in $(ii)$, $u_{k,i}$ is uniformly bounded in $\mathcal{C}^{0,\alpha}(\Omega)$ for every bounded open subset $\Omega$ of $\mathbb{R}^N$ and $\alpha\in (0,1)$. Consequently, by the Arzela-Ascoli theorem, $u_{k,i}\rightarrow u_{\infty,i}$ in $\mathcal{C}^{0,\alpha}(\Omega)$. This ensures that all hypotheses of \cite[Theorem C.1]{cpt} are satisfied. Therefore, $(iii)$ holds true.

$(iv):$ \ Consider $\ell=2$. In this case, $\varrho_2=\tau$. The nontrivial $G$-invariant solutions of problem \eqref{eq:single} that satisfy the condition 
\begin{align}\label{antisymetritau}
	u(\tau x)=-u(x) \quad \mbox{ for all } x\in \mathbb{R}^N
\end{align}
are the critical points of the functional $J$, defined in \eqref{functionalJsingle}, restricted to the Nehari manifold
\begin{align*}
	\mathcal{M}_V^{(G,\tau)}:= \{ u\in H_0^1(\rn)^G:u\neq 0, \ \|u\|_V^2=|u|_{2p}^{2p}, \ u\text{ satisfies }\eqref{antisymetritau}\}.
\end{align*}
Note that every such solution is nonradial and changes sign. There is a one-to-one correspondence 
\begin{align*}
	\mathcal{W}_V^G\to \mathcal{M}_V^{(G,\tau)},\quad (u_1,u_2)\mapsto u_1-u_2,
\end{align*}
whose inverse is $u\mapsto (u^+,u^{-})$, where $u^+=\max\{0,u\}$, $u^{-}=\min \{0,u\}$. It satisfies
\begin{align*}
	\mathcal{J}_V(u_1,u_2)=\dfrac{p-1}{2p}\left(\|u_1\|_V^2+\|u_2\|_V^2\right)= J(u_1-u_2).
\end{align*}
Therefore,
\begin{align*}
	J(u_{\infty,1}-u_{\infty,2})=\mathcal{J}_V((u_{\infty,1}, u_{\infty,2}))=\inf_{(u_1,u_2)\in \mathcal{W}_V^G}\mathcal{J}_V((u_1,u_2))=\inf_{u\in \mathcal{M}_V^{(G,\tau)}}J(u).
\end{align*}
Thus, $u_{\infty,1}-u_{\infty,2}\in H^{(G,\tau)}$ is a nonradial $G$-invariant sign-changing solution of \eqref{eq:single}.
\end{proof}
 
\begin{proof}[Proof of Theorem \ref{theorembetainfinity2}]
To find the points $y_k\in\r^{N-4}$ we argue as in Theorem \ref{thm:betatozero2}. Then, we follow the proof of Theorem \ref{theorembetainfinity1}.
\end{proof}

\bigskip

\begin{flushleft}
\textbf{Mónica Clapp} and \textbf{Víctor A. Vicente-Benítez}\\
Instituto de Matemáticas\\
Universidad Nacional Autónoma de México \\
Campus Juriquilla\\
Boulevard Juriquilla 3001\\
76230 Querétaro, Qro., Mexico\\
\texttt{monica.clapp@im.unam.mx} \\
\texttt{va.vicentebenitez@im.unam.mx} 
\medskip

\textbf{Mayra Soares}\\
Departamento de Matemática\\
Universidade de Brasília\\
Campus Darci Ribeiro\\
Asa Norte, Brasília\\ 70910-900, Brazil\\
\texttt{mayra.soares@unb.br}
\medskip

\textbf{Alberto Saldaña}\\
Instituto de Matemáticas\\
Universidad Nacional Autónoma de México \\
Circuito Exterior, Ciudad Universitaria\\
04510 Coyoacán, Ciudad de México, Mexico\\
\texttt{alberto.saldana@im.unam.mx}
\end{flushleft}
	
\end{document}